\pdfoutput=1
\documentclass[letterpaper]{amsart}

\usepackage[utf8]{inputenc}
\usepackage{lmodern}
\usepackage[T1]{fontenc}
\usepackage{amssymb}
\usepackage{enumitem}
\usepackage{mathtools}
\usepackage{xr}

\usepackage{graphicx}
\usepackage{tikz-cd}
\usepackage[pdfusetitle]{hyperref}
\tikzcdset{arrow style=Latin Modern}
\usepackage{cleveref}

\def\Fchoose#1#2{#1}

\mathtoolsset{mathic}
\DeclareMathAlphabet\mathbfit{OML}{cmm}{b}{it}

\newlist{enumarabic}{enumerate}{1}
\setlist[enumarabic]{font=\normalfont,label=(\arabic*),leftmargin=0.3in}
\newlist{enumroman}{enumerate}{1}
\setlist[enumroman]{font=\normalfont,label=(\roman*),leftmargin=0.3in}

\def\refhomog#1{\ref*{h@#1}}
\makeatletter
\def\eqrefhomog#1{\textup{\tagform@{\ref*{h@#1}}}}
\makeatother
\def\citehomog#1{\cite[#1]{Franz:2019}}

\def\cf{\emph{cf.}}

\def\arxiv#1{\href{http://arxiv.org/abs/#1}{\texttt{arXiv:#1}}}
\def\doi#1{\href{http://doi.org/#1}{doi:#1}}

\makeatletter
\g@addto@macro\bfseries{\boldmath}
\makeatother

\newcommand*\xbar[1]{%
   \hbox{%
     \vbox{%
       \hrule height 0.35pt 
       \kern0.35ex
       \hbox{%
         \kern-0.15em
         \ensuremath{#1}%
         \kern-0.15em
       }%
     }%
   }%
} 

\makeatletter
\def\smallunderbrace#1{\mathop{\vtop{\m@th\ialign{##\crcr
   $\hfil\displaystyle{#1}\hfil$\crcr
   \noalign{\kern3\p@\nointerlineskip}%
   \tiny\upbracefill\crcr\noalign{\kern3\p@}}}}\limits}
\makeatother

\let\newterm\emph

\def\Z{\mathbb Z}

\def\cupone{\mathbin{\cup_1}}
\def\cuptwo{\mathbin{\cup_2}}

\def\deg#1{|#1|}

\let\shuffle\nabla

\def\kk{\Bbbk}

\let\epsilon\varepsilon
\let\phi\varphi

\def\susp{\mathbf{s}}
\def\desusp{\susp^{-1}}

\def\cc{\mathbfit{c}}
\def\xx{\mathbfit{x}}

\def\aaa{\mathfrak{a}}
\def\bbb{\mathfrak{b}}

\def\AW{AW}

\def\EE{\mathbf{E}}

\def\FF{\mathbf{F}}

\def\ha{h^{a}}
\def\hc{h^{c}}

\let\emptyset\varnothing

\def\iter#1#2{#1^{[#2]}}

\def\BB{\mathbf{B}}

\def\HH{H\mkern-2mu H}

\let\KS\varkappa
\def\eqKS{\stackrel{\KS}{=}}
\def\noeq{\mathrel{\phantom{=}}}

\DeclareMathOperator{\Der}{Der}

\def\selfmapright{\mathbin{\rotatebox{90}{$\circlearrowleft$}}}

\def\hgaop{F_{2}\mathcal{X}}

\def\jamin{\mu}
\def\jamax{\nu}

\setlistdepth{8}
\newlist{caselist}{enumerate}{8}
\setlist[caselist]{label*={\bf\arabic*.},wide,leftmargin=0.1in}

\def\cancelswith#1{$\to$~\ref{#1}}

\newcommand{\nocontentsline}[3]{}
\newcommand{\tocless}[2]{\bgroup\let\addcontentsline=\nocontentsline{}\egroup}

\theoremstyle{plain}
\newtheorem{theorem}{Theorem}[section]
\newtheorem{proposition}[theorem]{Proposition}
\newtheorem{lemma}[theorem]{Lemma}
\newtheorem{corollary}[theorem]{Corollary}

\theoremstyle{definition}

\newtheorem{remark}[theorem]{Remark}

\theoremstyle{remark}
\newtheorem*{acknowledgements}{Acknowledgements}

\numberwithin{equation}{section}

\allowdisplaybreaks[4]

\def\kk{\Bbbk}
\def\Ai{$A_{\infty}$}
\def\cupone{\mathbin{\cup_1}}
\def\cuptwo{\mathbin{\cup_2}}

\begin{document}

\title[Homotopy Gerstenhaber algebras are shc]{Homotopy Gerstenhaber algebras\\are strongly homotopy commutative}
\author{Matthias Franz}
\thanks{The author was supported by an NSERC Discovery Grant.}
\address{Department of Mathematics, University of Western Ontario,
  London, Ont.\ N6A\;5B7, Canada}
\email{mfranz@uwo.ca}


\begin{abstract}
  We show that any homotopy Gerstenhaber algebra
  is naturally a strongly homotopy commutative (shc) algebra in the sense of Stasheff--Halperin
  with a homotopy-associative structure map.
  In the presence of certain additional operations corresponding
  to a \(\cupone\)-product on the bar construction,
  the structure map becomes homotopy commutative, so that one obtains an shc algebra in the sense of Munkholm.  
\end{abstract}

\maketitle

\section{Introduction}

Let \(A\) and~\(B\) be augmented dgas over a commutative ring~\(\kk\).
Recall that an \Ai~map~\(f\colon A\Rightarrow B\)
is a map of differential graded coalgebras~\(\BB A\to\BB B\) between the bar constructions of~\(A\) and~\(B\).
According to Stasheff--Halperin~\cite[Def.~8]{StasheffHalperin:1970},
the dga~\(A\) is a \newterm{strongly homotopy commutative (shc) algebra} if
\begin{enumroman}
\item
  \label{shc-shm}
  the multiplication map~\(\mu_{A}\colon A\otimes A\to A\) extends to an \Ai~morphism
  \begin{equation*}
    \Phi\colon A\otimes A \Rightarrow A
  \end{equation*}
  in the sense that the base component~\(\Phi_{(1)}\colon A\otimes A\to A\) of~\(\Phi\) equals \(\mu\).
\end{enumroman}

Munkholm~\cite[Def.~4.1]{Munkholm:1974} additionally requires the following:
\begin{enumroman}[resume]
\item
  \label{shc-unit}
  The map~\(\eta_{A}\colon \kk\to A\), \(1\mapsto 1\) is a unit for~\(\Phi\), that is,
  \begin{equation*}
    \Phi\circ(1_{A}\otimes\eta_{A}) = \Phi\circ(\eta_{A}\otimes 1_{A}) = 1_{A}. 
  \end{equation*}
\item
  \label{shc-ass}
  The \Ai~map~\(\Phi\) is homotopy associative, that is,
  \begin{equation*}
    \Phi\circ(\Phi\otimes 1_{A}) \simeq \Phi\circ(1_{A}\otimes\Phi)\colon A\otimes A\otimes A\Rightarrow A.
  \end{equation*}
\item
  \label{shc-com}
  The map~\(\Phi\) is homotopy commutative, that is,
  \begin{equation*}
    \Phi\circ T_{A,A} \simeq \Phi\colon A\otimes A\Rightarrow A.
  \end{equation*}
\end{enumroman}
Note that we write \(1_{A}\) for the identity map of~\(A\)
and \(T_{A,A}\colon A\otimes A\to A\otimes A\) for the the transposition of factors.
Also, the compositions and tensor products above are those of \Ai~maps.

\goodbreak

Using acyclic models,
Munkholm~\cite[Prop.~4.7]{Munkholm:1974} has constructed a natural shc structure
on the normalized singular cochain complex~\(C^{*}(X)\) of a space~\(X\).
Slightly earlier, Gugenheim--Munkholm~\cite{GugenheimMunkholm:1974} have given
a recursive formula for~\(\Phi\) in this case based on the Eilenberg--Zilber contraction, see \Cref{rem:ez}.

Singular cochain complexes are the main example of homotopy Gerstenhaber algebras (hgas)
besides the Hochschild cochains of associative algebras.
Recall that an hga is essentially an augmented dga~\(A\) with maps
\begin{equation*}
  E_{k}\colon A\otimes A^{\otimes k}\to A
\end{equation*}
for~\(k\ge1\) that induce a dga structure on the bar construction~\(\BB A\) compatible with the diagonal.
See \Cref{sec:def-hga} for a reformulation in terms of the identities the operations~\(E_{k}\) have to satisfy.
Kadeishvili~\cite{Kadeishvili:2003} has identified
certain additional operations on an hga~\(A\)
that allow to define a \(\cupone\)-product on~\(\BB A\). We call such an hga \newterm{extended}, see \Cref{sec:extended-hga}.
Singular cochain algebras are extended hgas.

As pointed out by Kadeishvili~\cite[Sec.~1.6]{Kadeishvili:2004},
it follows from general considerations that any hga admits
an shc structure in the sense of Stasheff--Halperin
such that the composition of~\(\BB\Phi\) with the shuffle map,
\begin{equation}
  \label{eq:comp-shuffle-BPhi}
  \BB A \otimes \BB A \stackrel{\shuffle}\longrightarrow \BB(A\otimes A) \xrightarrow{\BB\Phi} \BB A,
\end{equation}
is homotopic to the product on~\(\BB A\) determined by the hga structure.

Our main result is the following. In the companion paper~\cite{Franz:2019}
we apply it to determine the cohomology rings of a large class of homogeneous spaces.

\begin{theorem}
  \label{thm:hga-shc}
  Let \(A\) be an hga.
  \begin{enumroman}
  \item
    There is a canonical shc structure on~\(A\)
    satisfying properties~\ref{shc-shm},~\ref{shc-unit} and~\ref{shc-ass} of the definition above.
  \item
    If \(A\) is extended, then the shc structure additionally satisfies property~\ref{shc-com}.
  \item
    All structure maps commute with morphisms of (extended) hgas.
  \item 
    The composition~\eqref{eq:comp-shuffle-BPhi} is exactly the multiplication map on~\(\BB A\).
  \end{enumroman}
\end{theorem}

The \Ai~map~\(\Phi\) is defined in \Cref{sec:ai-map}
and the associativity and commutativity homotopies
in Sections~\ref{sec:def-ha} and \ref{sec:def-hc}.
The verification that they satisfy the required identities
is an elementary, but very lengthy computation that has been relegated to the appendices.
The claimed naturality will be obvious from the construction.
We conclude in \Cref{sec:poly} with results about \Ai~maps and shc maps from polynomial algebras to dgas and hgas, respectively.

\begin{acknowledgements}
  Maple and SymPy~\cite{Sympy} were used to check our formulas.
  I thank Chris Hall for his help with Python programming.
\end{acknowledgements}

\section{Preliminaries}

From now on, we call \Ai~maps \newterm{strongly homotopy multiplicative (shm)}
because this terminology pairs better with ``shc algebras''. We refer to the companion paper~\citehomog{Secs.~\refhomog{sec:tw}--\refhomog{sec:shm-tensor}}
for our general conventions, the definitions of twisting cochains, twisting cochain homotopies
and the corresponding families as well as those of shm maps and their homotopies, compositions and tensor products.
We work over an arbitrary commutative ring~\(\kk\) with unit.

We recall from~~\citehomog{Sec.~\refhomog{sec:tw}} our ``\(\eqKS\)''~notation which
alleviates us from explicitly specifying signs coming from the Koszul sign rule.
For example, if we define a map~\(F\colon A\otimes B\otimes C\to A'\otimes B'\)
between complexes by
\begin{align}
  F(a,b,c) &\eqKS f(c) \otimes g(a,b), \\
  \shortintertext{then we mean}
  F(a,b,c) &= (-1)^{(\deg{a}+\deg{b}+\deg{g})\deg{c}}\, f(c) \otimes g(a,b).
\end{align}
Composition of maps is distributed over tensor products. For instance,
\begin{equation}
  G(a,b) \eqKS f_{1}(f_{2}(a))\otimes g_{1}(g_{2}(b))
\end{equation}
defines the element
\begin{equation}
  \label{eq:convention-eqKS}
  G = f_{1}\,f_{2} \otimes g_{1}\,g_{2}
  = (-1)^{\deg{f_{2}}\deg{g_{1}}}\,(f_{1}\otimes g_{1})\,(f_{2}\otimes g_{2})
\end{equation}
in the endomorphism operad.

\section{Homotopy Gerstenhaber algebras}
\label{sec:hga}

Homotopy Gerstenhaber algebras were introduced by Voronov--Gerstenhaber~\cite[\S 8]{VoronovGerstenhaber:1995}.
For the convenience of the reader, we reproduce the definition of an (extended) homotopy Gerstenhaber algebra
from~\citehomog{Sec.~\refhomog{sec:hga}}.

\subsection{Definition of an hga}
\label{sec:def-hga}

A \newterm{homotopy Gerstenhaber algebra} (homotopy G-alge\-bra, \newterm{hga})
is an augmented dga~\(A\) with certain operations
\begin{equation}
  E_{k}\colon A \otimes A^{\otimes k}\to A,
  \qquad
  a\otimes b_{1}\otimes\dots\otimes b_{k}\mapsto E_{k}(a;b_{1},\dots,b_{k})
\end{equation}
of degree~\(\deg{E_{k}}=-k\) for~\(k\ge1\).
It is often convenient to use the additional operation~\(E_{0}=1_{A}\).
These operations satisfy the following properties.
\begin{enumroman}
\def\newtag#1{\tag{#1}}
\item
  All~\(E_{k}\) with~\(k\ge1\) take values in the augmentation ideal~\(\bar A\) and vanish
  if any argument is equal to~\(1\).
\setcounter{equation}{99}
  \begin{align}
    \newtag{ii}
    \label{eq:def-Ek-d}
    d(E_{k})(a;b_{\bullet})
    &\eqKS b_{1}\,E_{k-1}(a;b_{\bullet})
    + \sum_{m=1}^{k-1}(-1)^{m}\,E_{k-1}(a;b_{\bullet},b_{m}b_{m+1},b_{\bullet}) \\
    \notag &\qquad + (-1)^{k}\,E_{k-1}(a;b_{\bullet})\,b_{k}.
  \end{align}
  for all~\(k\ge1\) and all~\(a\),~\(b_{1}\),~\dots,~\(b_{k}\in A\).
  \begin{equation}
    \newtag{iii}
    \label{eq:Eprodfirstarg}
    E_{k}(a_{1}a_{2};b_{\bullet}) \eqKS \!\!\sum_{k_{1}+k_{2}=k}\!\! E_{k_{1}}(a_{1};b_{\bullet})\,E_{k_{2}}(a_{2};b_{\bullet})
  \end{equation}
  for~\(k\ge0\) and all~\(a_{1}\),~\(a_{2}\),~\(b_{1}\),~\dots,~\(b_{k}\in A\),
  where the sum is over all decompositions of~\(k\) into two non-negative integers.
  \begin{multline}
    \newtag{iv}
    \label{eq:formula-Ek-El}
    E_{l}(E_{k}(a;b_{\bullet});c_{\bullet}) \eqKS \\ \sum_{\substack{i_{1}+\dots+i_{k}+{}\\j_{0}+\dots+j_{k}=l}}
    \!\!(-1)^{\epsilon}\,
    E_{n}\bigl(a;\underbrace{c_{\bullet}}_{j_{0}},E_{i_{1}}(b_{1};c_{\bullet}),\underbrace{c_{\bullet}}_{j_{1}},
    \dots,\underbrace{c_{\bullet}}_{j_{k-1}},E_{i_{k}}(b_{k};c_{\bullet}),\underbrace{c_{\bullet}}_{j_{k}}\bigr),
  \end{multline}
  for all~\(k\),~\(l\ge0\) and all~\(a\),~\(b_{1}\),~\dots,~\(b_{k}\),~\(c_{1}\),~\dots,~\(c_{l}\in A\),
  where the sum is over all decompositions of~\(l\) into \(2k+1\)~non-negative integers,
\setcounter{equation}{1}
  \begin{equation}
    n = k + \sum_{t=0}^{k}j_{t}
  \qquad\text{and}\qquad
    \epsilon = \sum_{s=1}^{k}i_{s}\Bigl(k+\sum_{t=s}^{k}j_{t}\Bigr) + \sum_{t=1}^{k}t\,j_{t}.
  \end{equation}
\end{enumroman}
A \newterm{morphism of hgas} is a morphism~\(f\colon A\to B\) of augmented dgas
that is compatible with the hga operations in the obvious way.  

\begin{remark}
  \label{rem:hgaop-basis}
  By the properties~\eqref{eq:Eprodfirstarg} and~\eqref{eq:formula-Ek-El} as well as multilinearity, 
  we can rewrite any expression formed within an hga
  as a linear combination of terms~\(W\) such that no sums or scalar multiples occur inside~\(W\) and such that
  the first argument of any operation~\(E_{k}\) appearing in~\(W\)
  is a single variable and neither a product nor another hga operation.
  When we speak of the terms \newterm{appearing} in some expression within an hga,
  we mean the terms appearing in such an expansion.

  An expansion of this kind is actually unique and corresponds to a \(\kk\)-basis
  for the operad~\(\hgaop\) governing homotopy Gerstenhaber algebras,
  compare~\cite[Sec.~4]{McClureSmith:2003} and~\cite[\S 1.6.6]{BergerFresse:2004}.
\end{remark}

\subsection{Extended hgas}
\label{sec:extended-hga}

In~\cite{Kadeishvili:2003} Kadeishvili introduced the notion of an `extended hga'
as an hga~\(A\) defined over~\(\kk=\Z_{2}\) that admits certain additional operations~\(E^{i}_{kl}\). 
Based on this, he constructed \(\cup_{i}\)-products on~\(\BB A\) for all~\(i\ge1\).
We will only need the family~\(F_{kl} = E^{1}_{kl}\), but for coefficients in any~\(\kk\).
We therefore define an hga to be \emph{extended} if it comes with a family of operations
\begin{equation}
  F_{kl}\colon A^{\otimes k}\otimes A^{\otimes l}\to A
\end{equation}
of degree~\(\deg{F_{kl}}=-(k+l)\) for~\(k\),~\(l\ge 1\), satisfying the following conditions:
The values of all operations~\(F_{kl}\) lie in the augmentation ideal~\(\bar A\) and vanish if any argument is equal to~\(1\in A\).
The differential of~\(F_{kl}\) is given by
\begin{equation}
  \label{eq:d-Fkl}
  d(F_{kl})(a_{\bullet};b_{\bullet}) = A_{kl} + (-1)^{k}\,B_{kl}
\end{equation}
for all~\(a_{1}\),~\dots,~\(a_{k}\),~\(b_{1}\),~\dots,~\(b_{l}\in A\),
where
\begin{align}
  A_{1l} &= E_{l}(a_{1};b_{\bullet}), \\
  A_{kl} &\eqKS a_{1}\,F_{k-1,l}(a_{\bullet};b_{\bullet})
    + \sum_{i=1}^{k-1} (-1)^{i}\,F_{k-1,l}(a_{\bullet},a_{i}a_{i+1},a_{\bullet};b_{\bullet}) \\
  \notag &\qquad + \sum_{j=1}^{l} (-1)^{k}\, F_{k-1,j}(a_{\bullet};b_{\bullet})\,E_{l-j}(a_{k};b_{\bullet}) \\
\shortintertext{for~\(k\ge2\), and}
  B_{k1} &\eqKS -E_{k}(b_{1};a_{\bullet}), \\
  B_{kl} &\eqKS \sum_{i=0}^{k-1} E_{i}(b_{1};a_{\bullet})\,F_{k-i,l-1}(a_{\bullet};b_{\bullet})
    + \sum_{j=1}^{l-1} (-1)^{j}\,F_{k,l-1}(a_{\bullet};b_{\bullet},b_{j}b_{j+1},b_{\bullet}) \\
    \notag &\qquad + (-1)^{l}\,F_{k,l-1}(a_{\bullet};b_{\bullet})\,b_{l}
\end{align}
for~\(l\ge2\), \cf~\cite[Def.~2]{Kadeishvili:2003}.

The operation~\(\cuptwo=-F_{11}\) is a \newterm{\(\cuptwo\)-product} for~\(A\) in the sense that
\begin{equation}
  \label{eq:cuptwo}
  d(\cuptwo)(a;b) = a\cupone b +(-1)^{\deg{a}\deg{b}}\,b\cupone a
\end{equation}
for all~\(a\),~\(b\in A\).
This implies that the Gerstenhaber bracket in~\(H^{*}(A)\) is trivial,
compare \citehomog{eq.~\eqrefhomog{eq:def-gerstenhaber-bracket}}.

A morphism of extended hgas is a morphism of hgas that commutes with all operations~\(F_{kl}\), \(k\),~\(l\ge1\).

Cochain algebras of simplicial sets, in particular singular cochain algebras of topological spaces,
are naturally extended hgas, compare~\citehomog{Sec.~\refhomog{sec:cochains-hga}}.

\section{The shm map}
\label{sec:ai-map}

We define the family of maps
\begin{equation}
  \Phi_{(n)}\colon (A\otimes A)^{\otimes n}\to A
\end{equation}
by
\begin{equation}
  \label{eq:def-Phi-n}
  \Phi_{(n)}( a_{\bullet}\otimes b_{\bullet} ) \eqKS (-1)^{n-1} \!\!\!
  \sum_{\substack{j_{1}+\dots+j_{n}\\=n-1}} \!\!\! E_{j_{1}}(a_{1};b_{\bullet})\cdots E_{j_{n}}(a_{n};b_{\bullet})\,b_{n}
\end{equation}
for~\(n\ge0\), where the sum is over all decompositions of~\(n-1\) into \(n\)~non-negative integers such that
\begin{equation}
  \label{eq:seq-def-shc-tw}  
  \forall\; 1\le s\le n \qquad j_{1}+\dots+j_{s} < s.
\end{equation}
This condition means that the arguments of any term~\(E_{j_{s}}(a_{s};\dots)\) are \(b\)-variables with indices strictly smaller than~\(s\).
It implies \(j_{1}=0\), so that each summand starts with the variable~\(a_{1}\).
Omitting the arguments~\(a_{\bullet}\otimes b_{\bullet}=a_{1}\otimes b_{1}\),~\(a_{2}\otimes b_{2}\),~\dots,
the components of~\(\Phi\) look as follows in small degrees:
\begin{align}
  \Phi_{(1)} &\eqKS
a_{{1}}\, b_{{1}},
\\
\label{eq:def-Phi-2}
  \Phi_{(2)} &\eqKS
-a_{{1}}\, E_{{1}} ( a_{{2}};b_{{1}} ) \, b_{{2}},
  \\
  \Phi_{(3)} &\eqKS
a_{{1}}\, E_{{1}} ( a_{{2}};b_{{1}} ) \, E_{{1}} ( a_{{3}};b_{{2}} ) \, b_{{3}}
+ a_{{1}}\, a_{{2}}\, E_{{2}} ( a_{{3}};b_{{1}},b_{{2}} ) \, b_{{3}}.
\end{align}
For~\(n\ge1\) the number of summands in~\(\Phi_{(n)}\) is the Catalan number~\(C_{n-1}\).

\begin{proposition}
  \label{thm:hga-shc-tw}
  The~\(\Phi_{(n)}\) assemble to an shm map~\(\Phi\colon A\otimes A\Rightarrow A\)
  that satisfies properties~\ref{shc-shm} and~\ref{shc-unit} of the definition of an shc algebra.
\end{proposition}

\begin{proof}
  The verification of property~\ref{shc-shm} is a direct computation, see \Cref{sec:hga-shc-tw}.
  For property~\ref{shc-unit} we observe that the normalization condition for the hga operations
  implies that \(\Phi_{(n)}\) vanishes for~\(n>1\) if all~\(a_{i}\) or all~\(b_{j}\) equal \(1\).
  Similarly, \(\Phi_{(n)}\) vanishes for~\(n>1\) if~\(a_{i}\otimes b_{i}=1\otimes 1\) for some~\(i\),
  as required by the definition of a twisting family, see~\citehomog{eq.~\eqrefhomog{eq:tw-fam-1-bis}}:
  For~\(i<n\) we would have \(b_{i}=1\)
  as an argument to some \(E_{k}\)-term with~\(k\ge1\). For~\(i=n\) the term~\(E_{j_{n}}(a_{n};\dots)\)
  vanishes since \(a_{n}=1\) and \(j_{n}\ge1\) as there is at least one more argument, namely \(b_{n-1}\).
\end{proof}

\begin{remark}
  \label{rem:ez}
  Let \(X\) be a simplicial set.
  Gugenheim--Munkholm~\cite[Thm.~4.1\(_{*}\)]{GugenheimMunkholm:1974} have given a recursive formula
  for an extension of the cup product to an shm map~\(C^{*}(X)\otimes C^{*}(X)\Rightarrow C^{*}(X)\).
  It is based on 
  the Eilenberg--Zilber contraction
  \begin{equation}
    C_{*}(X)\otimes C_{*}(X) \overset{\shuffle}{\underset{\AW}{\rightleftarrows}} C_{*}(X\times X)
    \, {} \selfmapright \, h.
  \end{equation}
  Computer calculations suggest that this algorithm leads to our map~\(\Phi\)
  if one uses a slight modification of the classical homotopy~\(h\colon 1_{C_{*}(X\times X)}\simeq\shuffle\AW\) defined by Eilenberg--Mac\,Lane.
  That the hga structure on \(1\)-reduced simplicial sets can be defined via the Gugenheim--Munkholm formula
  follows from combining results of Hess--Parent--Scott--Tonks~\cite[Sec.~5]{HessEtAl:2006} and the author~\cite[App.~A]{Franz:szczarba2}.
\end{remark}

\begin{proposition}
  \label{thm:BA-product-Phi}
  The composition
  \begin{equation*}
    \BB A\otimes \BB A \stackrel{\shuffle}\longrightarrow \BB(A\otimes A) \xrightarrow{\BB\Phi} \BB A
  \end{equation*}
  coincides with the product on~\(\BB A\) given by the hga structure of~\(A\).
\end{proposition}

\begin{proof}
  We verify that the twisting cochain associated to the composition given above
  equals the twisting cochain~\(\EE\) corresponding to the multiplication on~\(\BB A\) as
  defined in~\citehomog{eq.~\eqrefhomog{eq:def-EE}}.
  Let us consider the components
  \begin{equation}
    \label{eq:product-component}
    A^{\otimes k}\otimes A^{\otimes l} \xrightarrow{(\desusp)^{\otimes n}}
    \BB_{k}A\otimes \BB_{l}A \stackrel{\shuffle}{\longrightarrow}
    \BB_{k+l}(A\otimes A) \stackrel{\Phi}{\longrightarrow}
    A
  \end{equation}
  with~\(k\),~\(l\ge0\) and~\(n=k+l\).
  
  A look at the formula for~\(\Phi_{(1)}\) shows that \eqref{eq:product-component}
  is the identity map of~\(A\) if \((k,l)=(1,0)\) or~\((0,1)\).
  Moreover, the map is zero if \(k\ne1\) and~\(l=0\),
  or if \(k=0\) and~\(l\ne1\) since for~\(n\ge2\)
  at least one term~\(E_{m}\) with~\(m\ge1\) is contained in~\(\Phi_{(n)}\)
  and this term vanishes if any argument equals \(1\).
  These cases are therefore verified.

  Now assume \(k\),~\(l\ge1\), and let
  \begin{equation}
    \cc' = \pm\bigl[\,a'_{1}\otimes b'_{1}\bigm|\dots\bigm|a'_{n}\otimes b'_{n}\,\bigr]
  \end{equation}
  be a term appearing in~\(\shuffle(\cc)\) such that \(\Phi(\cc')\) is non-zero, where
  \begin{equation}
    \cc = (\desusp)^{\otimes n}(a_{1}\otimes\dots\otimes a_{k}\otimes b_{1}\otimes\dots\otimes b_{l}).
  \end{equation}
  Because \(b'_{1}\),~\dots,~\(b'_{n-1}\) become arguments to \(E\)~terms,
  they cannot equal \(1\). Hence \(a'_{1}=\dots=a'_{n-1}=1\) and \(i_{n-1}=1\) in~\eqref{eq:seq-def-shc-tw}.
  This implies
  \begin{equation}
    \cc' = \pm\bigl[\,1\otimes b_{1}\bigm|\dots\bigm|1\otimes b_{l}\bigm|a_{1}\otimes 1\,\bigr],
  \end{equation}
  hence
  \begin{equation}
    \Phi\,\shuffle(\cc) =
    \begin{cases}
      \pm E_{n}(a_{1};b_{1},\dots,b_{l}) & \text{if \(k=1\),} \\
      0 & \text{if \(k>1\).}
    \end{cases}
  \end{equation}

  It remains to verify that the sign is \(+1\) in the case~\(k=1\).
  Write \(B=A\otimes A\) and
  \begin{equation}
    \cc'' = (1\otimes b_{1})\otimes\dots\otimes(1\otimes b_{l})\otimes(a_{1}\otimes1) \in B^{\otimes(l+1)}.
  \end{equation}
  The summand of~\(\shuffle(\cc)\)
  that is not annihilated by~\(\Phi\) is
  \begin{align}
    \cc' &= T_{\desusp B,(\desusp B)^{\otimes l}}\bigl(\desusp\otimes(\desusp)^{\otimes l}\bigr)(\cc) \\
    \notag &= (-1)^{l} \bigl(\desusp\otimes(\desusp)^{\otimes l}\bigr)\,T_{B,B^{\otimes l}}(\cc)
    = (-1)^{\epsilon}\,(\desusp)^{\otimes(l+1)}(\cc'')
  \end{align}
  where
  \begin{align}
    \cc'' &= (1\otimes b_{1})\otimes\dots\otimes(1\otimes b_{l})\otimes(a_{1}\otimes1), \\
    \epsilon &= l+\deg{a_{1}}(\deg{b_{1}}+\dots+\deg{b_{l}}).
  \end{align}
  It is mapped to
  \begin{align}
    \Phi\,\shuffle(\cc)
    &= (-1)^{\epsilon}\,\Phi\,(\desusp)^{\otimes(l+1)}(\cc'') \\
    \notag &= (-1)^{\epsilon}\,\Phi_{(n)}(\cc'')
    = E_{l}(a_{1};b_{1},\dots,b_{l}),
  \end{align}
  as desired.
\end{proof}

\begin{corollary}
  \label{thm:BA-product-Phi-general}
  The following diagram commutes for any~\(n\ge0\):
  \begin{equation*}
    \begin{tikzcd}
      (\BB A)^{\otimes n} \arrow{rd}[left,pos=0.6,inner sep=2.5ex]{\iter{\mu}{n}} \arrow{rr}{\iter{\shuffle}{n}} & & \BB(A^{\otimes n}) \arrow{ld}[pos=0.25]{\BB\iter{\Phi}{n}} \\
      & \BB A
    \end{tikzcd}
  \end{equation*}
\end{corollary}

Here we have written \(\iter{\shuffle}{n}\) and~\(\iter{\mu}{n}\) for the \(n\)-fold iterations of the shuffle map and the multiplication on~\(A\), which are both associative.
See~\citehomog{eq.~\eqrefhomog{eq:def-Phi-n}} for the definition of the iterations of the shm map~\(\Phi\).

\begin{proof}
  We proceed by induction.
  For \(n\le1\) there is nothing to show, and the case~\(n=2\) has been done above.
  For the induction step we observe that the parallelogram in the diagram
  \begin{equation*}
    \begin{tikzcd}[column sep=tiny]
      (\BB A)^{\otimes n} \otimes \BB A \arrow{rd}[left,pos=0.6,inner sep=2.5ex]{\iter{\mu}{n}\otimes1} \arrow{rr}[inner sep=0.75ex]{\iter{\shuffle}{n}\otimes1} & & \BB(A^{\otimes n})\otimes \BB A \arrow{ld}[pos=0.25]{\BB\iter{\Phi}{n}\otimes 1} \arrow{rr}[inner sep=0.75ex]{\shuffle} & & \BB(A^{\otimes n}\otimes A) \arrow{ld}[pos=0.25]{\BB(\iter{\Phi}{n}\otimes 1)} \\
      & \BB A \otimes \BB A \arrow{rd}[left,pos=0.74,inner sep=2.5ex]{\mu} \arrow{rr}[below,inner sep=0.5ex]{\shuffle} & & \BB(A\otimes A) \arrow{ld}[pos=0.4]{\BB\Phi} \\
      & & \BB A
    \end{tikzcd}
  \end{equation*}
  commutes by~\citehomog{Lemma~\refhomog{thm:shuffle-natural-shm}} and therefore the outer triangle by induction.
  This establishes the claim for~\(n+1\) and completes the proof.
\end{proof}

\section{Homotopy associativity}
\label{sec:def-ha}

The goal of this section is to establish a homotopy~\(\ha\)
between the twisting cochains~\(\Phi\circ(\Phi\otimes1)\) and~\(\Phi\circ(1\otimes\Phi)\).
To state our definition, we need to introduce some terminology.

A \newterm{\(c\)-product} is a product of one or more variables~\(c_{k}\);
it is called \newterm{proper} if it has more than one factor.
A \newterm{\(b\)-term} is a term of the form~\(E_{m}(b_{j};\dots)\)
with~\(m\ge0\) where all remaining arguments are \(c\)-products.
A \newterm{\(bc\)-product} is a product of one or more \(b\)-terms, say ending with~\(E_{m}(b_{j};\dots)\),
followed by the variable~\(c_{j}\).
An \newterm{\(a\)-term} is a term of the form~\(E_{m}(a_{i};\dots)\)
with~\(m\ge0\) where all remaining arguments are \(b\)-terms, \(c\)-products or \(bc\)-products.
An \newterm{\(ab\)-product} is a product of one or more \(a\)-terms and possibly \(b\)-terms that
ends with an \(a\)-term.
If we want to be more precise about the first variable of an \(E\)-term, we call it an \(a_{i}\)-term or a \(b_{j}\)-term.

To motivate our formula, we observe the following:
By~\citehomog{eq.~\eqrefhomog{eq:f-mult-otimes-g}} we have
\begin{equation}
  (1\otimes\Phi)_{(n)}(a_{\bullet}\otimes b_{\bullet}\otimes c_{\bullet})=a_{1}\cdots a_{n}\otimes\Phi_{(n)}(b_{\bullet}\otimes c_{\bullet}).
\end{equation}
Note that \(\Phi_{(n)}(b_{\bullet}\otimes c_{\bullet})\) is a \(bc\)-product.
To compute \((\Phi\circ(1\otimes\Phi))_{(n)}(a_{\bullet}\otimes b_{\bullet}\otimes c_{\bullet})\) we use
\citehomog{eq.~\eqrefhomog{eq:twc-composition}}. Taking property~\eqref{eq:Eprodfirstarg} of the definition of an hga
into account, we see that we obtain
a sum of terms~\(\pm\,U\,V\) where \(V\) is a \(bc\)-product and \(U\) a product of \(a\)-terms
having only \(bc\)-products as arguments. A similar argument, combined with the associativity condition~\eqref{eq:formula-Ek-El},
shows that each term appearing in~\((\Phi\circ(\Phi\otimes1))_{(n)}\linebreak[1](a_{\bullet}\otimes b_{\bullet}\otimes c_{\bullet})\)
is of the form~\(\pm\,U\,b_{n}\,W\) where \(U\) is an \(ab\)-product and \(W\) a \(c\)-product. Moreover
no \(bc\)-products appear inside \(a\)-terms in this case, and the final variable of each \(c\)-product, say \(c_{j}\),
corresponds to a \(b_{j}\)-term that appears as a factor of the top-level product and not as an argument to some \(a\)-term.
The homotopy~\(\ha\) has to interpolate between the two kinds of terms we have described.

We set \(\ha_{(0)}=\eta_{A}\). For~\(n\ge1\), we define
\begin{equation}
  \label{eq:def-ha-UV}
  \ha_{(n)} \eqKS \sum (-1)^{\epsilon}\,U\, V
\end{equation}
where the sum is over all \(ab\)-products~\(U\) and all \(bc\)-products~\(V\) satisfying the following conditions:
\begin{enumroman}
\item 
  \label{ass-cond-1}
  Each of the \(3n\)~variables~\(a_{1}\),~\dots,~\(c_{n}\) appears exactly once in~\(U\,V\).
  The \(a_{i}\)'s appear in ascending order, as do the \(b_{i}\)'s and the \(c_{i}\)'s.
  Moreover, \(a_{i}\) precedes \(b_{i}\) and \(b_{i}\)~precedes \(c_{i}\) for each~\(i\).
\item \label{ass-cond-2}
  The first argument to any \(E\)-term in~\(U\,V\) has a larger index than the remaining arguments.
  (The definitions above imply that the first argument also has a smaller letter than the remaining arguments, where \(a<b<c\).)
\item \label{ass-cond-3}
  Any top-level \(b\)-term appearing in~\(U\), say with first argument~\(b_{i}\),
  comes right after the \(a\)-term with first argument~\(a_{i}\).
\item Define
  \begin{align}
    J_{a} &= \bigl\{\, j \bigm| \text{\(b_{j}\) appears in a \(b\)-term that is argument to an \(a\)-term}\,\bigr\}, \\
    J_{b} &= \bigl\{\, j \bigm| \text{\(b_{j}\) appears in a top-level \(b\)-term inside the \(ab\)-product~\(U\)}\,\bigr\}, \\
    J_{c} &= \bigl\{\, j \bigm| \text{\(b_{j}\) appears in a \(bc\)-product}\,\bigr\}.
  \end{align}
  By construction,  \(\{1,\dots,n\}\) is the disjoint union of~\(J_{a}\),~\(J_{b}\) and~\(J_{c}\).
  Note that \(J_{c}\) cannot be empty as \(V\) is a \(bc\)-product.
  We additionally require
  \begin{gather}
    J_{a}\ne\emptyset,
    \qquad\qquad
    J_{a} \cup J_{b} = \{\,1,\dots,\jamax\,\}, \\
    J_{a}\setminus\{\jamax\} = \bigl\{\, j \bigm| \text{\(c_{j}\) appears in a \(c\)-product, but not as last factor}\,\bigr\}
  \end{gather}
  where we have written \(\jamax=\max J_{a}\).
  If \(c_{j}\) does not appear in a proper product, then it is considered to be the last factor of a \(c\)-product with a single factor.
  Consequently, we have \(n\ge2\) and \(J_{c} = \{\jamax+1,\dots,n\}\).
\end{enumroman}

The sign exponent~\(\epsilon\) in~\eqref{eq:def-ha-UV} is defined recursively.
Write \(\jamin=\min J_{a}\). If \(\jamin=\jamax\), that is, if \(J_{a}=\{\jamax\}\), then
\begin{align}
  \epsilon &= n + \text{contribution of the \(b\)-term~\(E(b_{\jamax};\dots)\)} \\
  \notag &\qquad + \text{contribution of each \(bc\)-product occurring inside an \(a\)-term}.
\end{align}
The contributions are as follows: Consider a \(bc\)-product
\begin{equation}
  E_{q_{j}}(b_{j};\dots)\cdots E_{q_{k}}(b_{k};\dots)\,c_{k}
\end{equation}
occurring in some term~\(E_{p}(a_{i};\dots)\) as \(m\)-th argument (with~\(a_{i}\) being at position~\(0\) and the last argument at position~\(p\)).
The contribution of such a term is
\begin{equation}
  (q_{j}+\dots+q_{k})(p-m+1).
\end{equation}
Note that \(q_{j}+\dots+q_{k}\) is the degree of the \(bc\)-product, considered as a function of its arguments,
and \(p-m+1\) is the number of arguments of~\(E_{p}(a_{i};\dots)\) from the \(bc\)-product (including) to the end.
If \(E_{q}(b_{\jamax};\dots)\) occurs in the term~\(E_{p}(a_{i};\dots)\) as \(m\)-th argument, then its contribution is
\begin{equation}
  \hat\epsilon + m + q\,(p-m)
\end{equation}
where \(\hat\epsilon\) is the degree of the terms preceding the \(E(a_{i};\dots)\)-term, again considered
as a function of their arguments. For example, if \(U\,V\) is
\begin{equation}
  a_{1}\,b_{1}\,a_{2}\,E_{1}(b_{2};c_{1})\,a_{3}\,a_{4}\,a_{5}\,E_{2}\bigl(a_{6};E_{1}(b_{3};c_{2}),b_{4}\,E_{2}(b_{5};c_{3},c_{4})\,c_{5}\bigr)\,b_{6}\,c_{6},
\end{equation}
then the contribution of the \(bc\)-product inside the \(a_{6}\)-term is \(2\cdot 1=2\),
and that of the \(b_{3}\)-term is \(1+1+1\cdot 1=3\).

If \(J_{a}\) is not a singleton, then we compare \eqref{eq:def-ha-UV} to a summand~\((-1)^{\epsilon'}\,U'\,V'\)
with~\(J_{a}'=J_{a}\setminus\{\jamin\}\) and \(J_{b}'=J_{b}\cup\{\jamin\}\). More precisely: We can write \(U\,V\) as
\begin{multline}
  \label{eq:ha-sign-def-rec-term}
  \allowdisplaybreaks[1]
  E_{p_{1}}(a_{1};c_{\bullet})\,E_{q_{1}}(b_{1};c_{\bullet})\cdots E_{q_{\jamin-1}}(b_{\jamin-1},c_{\bullet})\,E_{p_{\jamin}}(a_{\jamin};c_{\bullet})\\
  {}\cdot E_{p_{\jamin+1}}(a_{\jamin+1};c_{\bullet})\cdots E_{p_{i}}(a_{i};
  c_{\bullet},E_{q}(b_{\jamin};c_{\bullet}),\dots)\cdots c_{n}
\end{multline}
where the \(b_{\jamin}\)-term is the \(m\)-th argument of the \(a_{i}\)-term.
We define \(U'\,V'\) as
\begin{multline}
  E_{p_{1}}(a_{1};c_{\bullet})\,E_{q_{1}}(b_{1};c_{\bullet})\cdots E_{q_{\jamin-1}}(b_{\jamin-1},c_{\bullet})\,E_{p_{\jamin}+\dots+p_{i-1}
  +m-1}(a_{\jamin};c_{\bullet})\\
  {}\cdot E_{q}(b_{\jamin};c_{\bullet})\,a_{\jamin+1}\cdots a_{i-1}\,E_{p_{i}
  -m}(a_{i};\dots)\cdots c_{n}
\end{multline}
where all \(c\)-variables~\(c_{\bullet}\) appearing between~\(a_{\jamin+1}\) and~\(b_{\jamin}\) have been moved as additional arguments
to the term~\(E(a_{\jamin};\dots)\). Moreover, the proper \(c\)-product starting with~\(c_{\jamin}\) is split
into~\(c_{\jamin}\) and the remaining product. These two arguments replace the original \(c\)-product, wherever
it appears in~\(U\,V\). For example, if \(U\,V\) is
\begin{equation}
  a_{1}\,b_{1}\,a_{2}\,b_{2}\,a_{3}\,E_{1}(a_{4};c_{1})\,E_{2}(a_{5};E_{1}(b_{3},c_{2}),b_{4})\,E_{1}(a_{6};b_{5})\,E_{1}(b_{6};c_{3}\,c_{4}\,c_{5})\,c_{6},
\end{equation}
then \(U'\,V'\) equals
\begin{equation}
  \label{eq:ex-U'V'}
  a_{1}\,b_{1}\,a_{2}\,b_{2}\,E_{1}(a_{3};c_{1})\,E_{1}(b_{3},c_{2})\,a_{4}\,E_{1}(a_{5};b_{4})\,E_{1}(a_{6};b_{5})\,E_{2}(b_{6};c_{3},c_{4}\,c_{5})\,c_{6}.
\end{equation}

The sign exponents~\(\epsilon\) and~\(\epsilon'\) are related by
\begin{equation}
  \label{eq:ha-sign-def-rec}
  \epsilon' - \epsilon =  \tilde\epsilon + m + q\,(p_{i} -m) + \hat\epsilon,
\end{equation}
where
\begin{equation}
  \tilde\epsilon = \sum_{s=1}^{\jamin-1}(p_{s}+q_{s}) + \sum_{s=\jamin}^{i-1}p_{s}
\end{equation}
is the degree of the expression preceding \(E(a_{i};\dots)\)
as a function of its arguments, and
\(\hat\epsilon\) is the sign exponent for the summand of~\(d(U'V')\)
that recombines \(c_{\jamin}\) and the following \(c\)-product to the original one.
In the example~\eqref{eq:ex-U'V'} we have \(\hat\epsilon=5\).
Since \(\tilde\epsilon\) is a summand of~\(\hat\epsilon\), the difference~\(\epsilon-\epsilon'\)
is actually independent of~\(\tilde\epsilon\).

\def\nm#1{\hbox{\normalsize$#1$}}
\def\sm#1{\clap{\small$#1$}}
\def\jcols#1#2#3{&\qquad\qquad\sm{#1}\quad&\qquad\sm{#2}\qquad&\quad\sm{#3}}
Omitting the arguments,~\(a_{\bullet}\otimes b_{\bullet}\otimes c_{\bullet}\),
the components of~\(\ha\) look as follows in small degrees.
We also indicate the values of~\(J_{a}\),~\(J_{b}\) and~\(J_{c}\) for each term.
Note that the vanishing of~\(\ha_{(1)}\) reflects the identity \((\Phi\circ(\Phi\otimes1))_{(1)}=(\Phi\circ(1\otimes\Phi))_{(1)}=\iter{\mu_{A}}{3}\).
\begin{alignat}{3}
  \ha_{(1)} &= 0, \jcols{\text{\normalsize\(J_{a}\)}}{\text{\normalsize\(J_{b}\)}}{\text{\normalsize\(J_{c}\)}} \\
  \ha_{(2)} &\eqKS
{} - a_{{1}}\, E_{{2}} ( a_{{2}};b_{{1}},c_{{1}} ) \, b_{{2}}\, c_{{2}} \jcols{\{1\}}{\emptyset}{\{2\}}
\\ \notag &\noeq {} - a_{{1}}\, E_{{1}} ( a_{{2}};b_{{1}} ) \, E_{{1}} ( b_{{2}};c_{{1}} ) \, c_{{2}}, \jcols{\{1\}}{\emptyset}{\{2\}} \\
  \ha_{(3)} &\eqKS
{} + a_{{1}}\, E_{{2}} ( a_{{2}};b_{{1}},c_{{1}} ) \, E_{{1}} ( a_{{3}};b_{{2}}\, c_{{2}} ) \, b_{{3}}\, c_{{3}} \jcols{\{1\}}{\emptyset}{\{2,3\}}
\\ \notag &\noeq {} - a_{{1}}\, E_{{1}} ( a_{{2}};b_{{1}} ) \, E_{{1}} ( a_{{3}};E_{{1}} ( b_{{2}};c_{{1}} ) \, c_{{2}} ) \, b_{{3}}\, c_{{3}} \jcols{\{1\}}{\emptyset}{\{2,3\}}
\\ \notag &\noeq {} + a_{{1}}\, E_{{1}} ( a_{{2}};b_{{1}} ) \, E_{{2}} ( a_{{3}};c_{{1}},b_{{2}}\, c_{{2}} ) \, b_{{3}}\, c_{{3}} \jcols{\{1\}}{\emptyset}{\{2,3\}}
\\ \notag &\noeq {} - a_{{1}}\, a_{{2}}\, E_{{2}} ( a_{{3}};b_{{1}},E_{{1}} ( b_{{2}};c_{{1}} ) \, c_{{2}} ) \, b_{{3}}\, c_{{3}} \jcols{\{1\}}{\emptyset}{\{2,3\}}
\\ \notag &\noeq {} + a_{{1}}\, a_{{2}}\, E_{{3}} ( a_{{3}};b_{{1}},c_{{1}},b_{{2}}\, c_{{2}} ) \, b_{{3}}\, c_{{3}} \jcols{\{1\}}{\emptyset}{\{2,3\}}
\\ \notag &\noeq {} + a_{{1}}\, E_{{2}} ( a_{{2}};b_{{1}},c_{{1}} ) \, a_{{3}}\, b_{{2}}\, E_{{1}} ( b_{{3}};c_{{2}} ) \, c_{{3}} \jcols{\{1\}}{\emptyset}{\{2,3\}}
\\ \notag &\noeq {} + a_{{1}}\, E_{{1}} ( a_{{2}};b_{{1}} ) \, E_{{1}} ( a_{{3}};c_{{1}} ) \, b_{{2}}\, E_{{1}} ( b_{{3}};c_{{2}} ) \, c_{{3}} \jcols{\{1\}}{\emptyset}{\{2,3\}}
\\ \notag &\noeq {} + a_{{1}}\, E_{{1}} ( a_{{2}};b_{{1}} ) \, a_{{3}}\, E_{{1}} ( b_{{2}};c_{{1}} ) \, E_{{1}} ( b_{{3}};c_{{2}} ) \, c_{{3}} \jcols{\{1\}}{\emptyset}{\{2,3\}}
\\ \notag &\noeq {} + a_{{1}}\, E_{{1}} ( a_{{2}};b_{{1}} ) \, a_{{3}}\, b_{{2}}\, E_{{2}} ( b_{{3}};c_{{1}},c_{{2}} ) \, c_{{3}} \jcols{\{1\}}{\emptyset}{\{2,3\}}
\\ \notag &\noeq {} + a_{{1}}\, a_{{2}}\, E_{{2}} ( a_{{3}};b_{{1}},c_{{1}} ) \, b_{{2}}\, E_{{1}} ( b_{{3}};c_{{2}} ) \, c_{{3}} \jcols{\{1\}}{\emptyset}{\{2,3\}}
\\ \notag &\noeq {} + a_{{1}}\, a_{{2}}\, E_{{1}} ( a_{{3}};b_{{1}} ) \, E_{{1}} ( b_{{2}};c_{{1}} ) \, E_{{1}} ( b_{{3}};c_{{2}} ) \, c_{{3}} \jcols{\{1\}}{\emptyset}{\{2,3\}}
\\ \notag &\noeq {} + a_{{1}}\, a_{{2}}\, E_{{1}} ( a_{{3}};b_{{1}} ) \, b_{{2}}\, E_{{2}} ( b_{{3}};c_{{1}},c_{{2}} ) \, c_{{3}} \jcols{\{1\}}{\emptyset}{\{2,3\}}
\\ \notag &\noeq {} - a_{{1}}\, b_{{1}}\, E_{{1}} ( a_{{2}};c_{{1}} ) \, E_{{2}} ( a_{{3}};b_{{2}},c_{{2}} ) \, b_{{3}}\, c_{{3}} \jcols{\{2\}}{\{1\}}{\{3\}}
\\ \notag &\noeq {} - a_{{1}}\, b_{{1}}\, E_{{1}} ( a_{{2}};c_{{1}} ) \, E_{{1}} ( a_{{3}};b_{{2}} ) \, E_{{1}} ( b_{{3}};c_{{2}} ) \, c_{{3}} \jcols{\{2\}}{\{1\}}{\{3\}}
\\ \notag &\noeq {} + a_{{1}}\, b_{{1}}\, a_{{2}}\, E_{{3}} ( a_{{3}};b_{{2}},c_{{1}},c_{{2}} ) \, b_{{3}}\, c_{{3}} \jcols{\{2\}}{\{1\}}{\{3\}}
\\ \notag &\noeq {} - a_{{1}}\, b_{{1}}\, a_{{2}}\, E_{{3}} ( a_{{3}};c_{{1}},b_{{2}},c_{{2}} ) \, b_{{3}}\, c_{{3}} \jcols{\{2\}}{\{1\}}{\{3\}}
\\ \notag &\noeq {} - a_{{1}}\, b_{{1}}\, a_{{2}}\, E_{{2}} ( a_{{3}};E_{{1}} ( b_{{2}};c_{{1}} ) ,c_{{2}} ) \, b_{{3}}\, c_{{3}} \jcols{\{2\}}{\{1\}}{\{3\}}
\\ \notag &\noeq {} + a_{{1}}\, b_{{1}}\, a_{{2}}\, E_{{2}} ( a_{{3}};b_{{2}},c_{{1}} ) \, E_{{1}} ( b_{{3}};c_{{2}} ) \, c_{{3}} \jcols{\{2\}}{\{1\}}{\{3\}}
\\ \notag &\noeq {} - a_{{1}}\, b_{{1}}\, a_{{2}}\, E_{{2}} ( a_{{3}};c_{{1}},b_{{2}} ) \, E_{{1}} ( b_{{3}};c_{{2}} ) \, c_{{3}} \jcols{\{2\}}{\{1\}}{\{3\}}
\\ \notag &\noeq {} + a_{{1}}\, b_{{1}}\, a_{{2}}\, E_{{1}} ( a_{{3}};E_{{1}} ( b_{{2}};c_{{1}} )  ) \, E_{{1}} ( b_{{3}};c_{{2}} ) \, c_{{3}} \jcols{\{2\}}{\{1\}}{\{3\}}
\\ \notag &\noeq {} + a_{{1}}\, b_{{1}}\, a_{{2}}\, E_{{1}} ( a_{{3}};b_{{2}} ) \, E_{{2}} ( b_{{3}};c_{{1}},c_{{2}} ) \, c_{{3}} \jcols{\{2\}}{\{1\}}{\{3\}}
\\ \notag &\noeq {} - a_{{1}}\, E_{{1}} ( a_{{2}};b_{{1}} ) \, E_{{2}} ( a_{{3}};b_{{2}},c_{{1}}\, c_{{2}} ) \, b_{{3}}\, c_{{3}} \jcols{\{1,2\}}{\emptyset}{\{3\}}
\\ \notag &\noeq {} - a_{{1}}\, E_{{1}} ( a_{{2}};b_{{1}} ) \, E_{{1}} ( a_{{3}};b_{{2}} ) \, E_{{1}} ( b_{{3}};c_{{1}}\, c_{{2}} ) \, c_{{3}} \jcols{\{1,2\}}{\emptyset}{\{3\}}
\\ \notag &\noeq {} - a_{{1}}\, a_{{2}}\, E_{{3}} ( a_{{3}};b_{{1}},b_{{2}},c_{{1}}\, c_{{2}} ) \, b_{{3}}\, c_{{3}} \jcols{\{1,2\}}{\emptyset}{\{3\}}
\\ \notag &\noeq {} - a_{{1}}\, a_{{2}}\, E_{{2}} ( a_{{3}};b_{{1}},b_{{2}} ) \, E_{{1}} ( b_{{3}};c_{{1}}\, c_{{2}} ) \, c_{{3}}. \jcols{\{1,2\}}{\emptyset}{\{3\}}
\end{alignat}

\begin{remark}
  The number of terms seems to grow rapidly with~\(n\), by a factor close to~\(10\) with each degree.
  There are no terms in~\(\ha_{(1)}\), two terms in~\(\ha_{(2)}\), \(25\)~terms in~\(\ha_{(3)}\),
  \(254\)~terms in~\(\ha_{(4)}\), \(2421\)~terms in~\(\ha_{(5)}\), \(22{,}522\)~terms in~\(\ha_{(6)}\),
  \(207{,}682\)~terms in~\(\ha_{(7)}\), \(1{,}911{,}954\)~terms in~\(\ha_{(8)}\) and \(17{,}635{,}830\)~terms in~\(\ha_{(9)}\).
\end{remark}

\begin{proposition}
  \label{thm:hga-shc-homass}
  The maps~\(\ha_{(n)}\) assemble to an shm homotopy~\(\ha\) from~\(\Phi\circ(\Phi\otimes1)\) to~\(\Phi\circ(1\otimes\Phi)\).
\end{proposition}

\begin{proof}
  This is a very long direct computation, see \Cref{sec:hga-shc-homass}.
  We remark that this proof is the only place in this paper where we use
  the associativity condition~\eqref{eq:formula-Ek-El} for the hga structure,
  besides assuming that the product on~\(\BB A\) is associative.

  Let us verify here the normalization condition~\citehomog{eq.~\eqrefhomog{eq:tw-h-family-1}}
  for twisting homotopy families.
  Consider a term~\(U\,V\) in~\eqref{eq:def-ha-UV} 
  and assume that \(a_{i}=b_{i}=c_{i}=1\) for some~\(i\).
  
  If \(i\in J_{a}\), then the \(a\)-term containing \(b_{i}\) vanishes and therefore also \(U\,V\).

  In the case~\(i\in J_{b}\) the \(b_{i}\)-term is top-level and there is a \(c\)-product ending in~\(c_{i}\).
  If \(i>1\) and \(b_{i-1}\) is not top-level, then it must be an argument to the \(a_{i}\)-term, which therefore vanishes.
  Otherwise, the \(c\)-product containing \(c_{i}\) is just \(c_{i}=1\) itself. Since it is argument to some \(a\)-term or \(b\)-term,
  the whole expression is again \(0\).

  We finally consider the case \(i\in J_{c}\). Again, the \(c\)-product containing \(c_{i}\) is \(c_{i}=1\) itself.
  If the \(bc\)-term containing \(b_{i}\) does not end in~\(c_{i}\), then
  \(c_{i}\) is argument to some later \(b\)-term in the same \(bc\)-product, so that \(U\,V\) vanishes.
  If the \(bc\)-product ends in~\(c_{i}\), we look at \(b_{i-1}\).
  If it appears in the same \(bc\)-product, then \(c_{i-1}\) is an argument of the \(b_{i}\)-term.
  If \(i=\jamax\) or if \(b_{i-1}\) is part of an earlier \(bc\)-product, then it appears inside the \(a_{i}\)-term,
  which once again forces \(U\,V\) to vanish.  
\end{proof}

\begin{proposition}
  \label{thm:Bha-shuffle-vanish}
  The associated coalgebra homotopy~\(\BB\ha\) vanishes on the image
  of the iterated shuffle map~\(\iter{\shuffle}{3}\colon \BB A\otimes\BB A\otimes\BB A\to\BB(A\otimes A\otimes A)\).
\end{proposition}

\begin{proof}
  Assume that \(\ha\) does not vanish on the term
  \begin{equation}
    \pm\bigl[ a'_{1}\otimes b'_{1}\otimes c'_{1}\bigm|\dots\bigm| a'_{n}\otimes b'_{n}\otimes c'_{n} \bigr]
  \end{equation}
  appearing in
  \begin{equation}
    \iter{\shuffle}{3}\bigl( [a_{1}|\dots|a_{k}]\otimes[b_{1}|\dots|b_{l}]\otimes[c_{1}|\dots|c_{m}]\bigr)
  \end{equation}
  where \(n=k+l+m\ge1\). 
  By the definition of the shuffle map, two of the three factors are equal to~\(1\) in each tensor product~\(a'_{j}\otimes b'_{j}\otimes c'_{j}\).
  We must have \(b'_{j}\ne1\) for~\(j\in J_{a}\) because
  otherwise the \(b'_{j}\)-term would have value~\(0\) or~\(1\), in which case the surrounding \(a\)-term vanishes.
  This implies \(c'_{j}=1\) for all~\(j\in J_{a}\), so that the \(c\)-product~\(\tilde c\) ending in~\(c_{\jamax}\) is \(1\).
  As a consequence, the \(E\)-term having \(\tilde c\) as an argument vanishes.
  We conclude that for~\(n\ge1\) there is no term in the image of the shuffle map on which \(\ha\) assumes a non-zero value.
  In other words, \(\BB\ha\,\iter{\shuffle}{3}=0\).
\end{proof}

\begin{remark}
  \Cref{thm:BA-product-Phi-general} applies in particular to~\(\iter{\Phi}{3}=\Phi\circ(\Phi\otimes1)\),
  and a look at the proof shows that the conclusion equally holds for~\(\Phi\circ(1\otimes\Phi)\).
  Together with \Cref{thm:Bha-shuffle-vanish}, this implies (in a quite roundabout way)
  that condition~\eqref{eq:formula-Ek-El} of an hga structure
  indeed leads to an associative multiplication on the bar construction.
\end{remark}

In the remainder of this section we collect some observations that will be used in \Cref{sec:hga-shc-homass}.
We start by noting the following variant of property~\eqref{eq:def-Ek-d} of the definition of an hga,
\begin{equation}
  \label{eq:d-E-E}
  d(E_{p})\bigl(a;E_{q}(b;c_{\bullet}),c_{\bullet}\bigr)
  \eqKS (-1)^{q(p-1)}\,E_{q}(b;c_{\bullet})\,E_{p-1}(a;c_{\bullet}) + \cdots
\end{equation}
valid for all~\(k\),~\(q\ge0\) and all~\(a\),~\(b\),~\(c_{\bullet}\in A\).
This is a consequence of the convention~\eqref{eq:convention-eqKS}
because we have permuted the operations~\(E_{p-1}\) and~\(E_{q}\).
Similarly, property~\eqref{eq:Eprodfirstarg} implies
\begin{multline}
  \label{eq:E-prod-E}
  E_{k}\bigl(E_{p_{1}}(a_{1};b_{\bullet})\,E_{p_{2}}(a_{2};b_{\bullet});c_{\bullet}\bigr) \eqKS \\
  \sum_{k_{1}+k_{2}=k}\!\! (-1)^{p_{1}k_{2}}\,
  E_{k_{1}}\bigl(E_{p_{1}}(a_{1};b_{\bullet});c_{\bullet}\bigr)\,
  E_{k_{2}}\bigl(E_{p_{2}}(a_{2};b_{\bullet});c_{\bullet}\bigr)
\end{multline}
for all~\(k\),~\(p_{1}\),~\(p_{2}\ge0\) and all~\(a_{1}\),~\(a_{2}\),~\(b_{\bullet}\),~\(c_{\bullet}\in A\).

Let us write \(\Phi'=\Phi\circ(\Phi\otimes1)\).
We note that each term appearing \(\Phi'(a_{\bullet}\otimes b_{\bullet}\otimes c_{\bullet})\)
contains only one \(bc\)-product (at the very end) and that no two top-level \(b\)-terms are adjacent.

\begin{lemma}
  \label{thm:sign-change-Y}
  Let \((-1)^{\epsilon}\,W\) be a summand appearing in~\(\Phi'_{(n)}(a_{\bullet}\otimes b_{\bullet}\otimes c_{\bullet})\).
  Assume that it has at least one \(b\)-term inside an \(a\)-term,
  and let \(\mu\) be the smallest index of such a \(b\)-variable.
  Then a term~\(E_{q}(b_{\mu};\dots)\) appears as, say,
  the \(m\)-th argument~ of a term~\(E_{p_{r}}(a_{r};\dots)\),
  and \(c_{\mu}\) as the first variable inside a proper \(c\)-product~\(c_{\mu}\tilde c\).
  Let \(W'\) be obtained from~\(W\) in the same way as in~\eqref{eq:ha-sign-def-rec-term},
  and let \(\epsilon'\) be the sign exponent of~\(W'\) in~\(\Phi'_{(n)}(a_{\bullet}\otimes b_{\bullet}\otimes c_{\bullet})\).
  We have
  \begin{equation*}
    \epsilon'-\epsilon \equiv \tilde\epsilon + m + q\,(p_{i}-1) + \hat\epsilon \pmod{2}.
  \end{equation*}
  Here \(\tilde\epsilon\) is the degree of all \(E\)-operations in front of the \(a_{i}\)-term in~\(W\),
  and \(\hat\epsilon\) is the sign exponent of the part of the differential~\(d(W')\) that recombines \(c_{\mu}\) and~\(\tilde c\) to~\(c_{\mu}\tilde c\).
\end{lemma}

In other words, the rule used in the recursive sign definition~\eqref{eq:ha-sign-def-rec} for~\(\ha\)
also applies to~\(\Phi'\).
Note that \(c_{\mu}\tilde{c}\) may be the trailing \(c\)-product in the case of~\(\Phi'\), which is impossible for~\(\ha\).
This will be important for the pair~\ref{ha-s:36.2.2}\ in \Cref{sec:hga-shc-homass-signs}.

\begin{proof}
  We can perform the modification~\(W\to W'\) in steps: We move \(c\)-variables in front of the \(b_{\mu}\)-term
  to the preceding \(a\)-term (as in the pair~\ref{ha-s:13}\ in \Cref{sec:hga-shc-homass-pairs}),
  we move the \(b_{\mu}\)-term to the preceding \(a\)-term (as in the pair~\ref{ha-s:1})
  or out of the current \(a\)-term if it is contained in the \(a_{\mu+1}\)-term (as in~\ref{ha-s:2} and~\ref{ha-s:4}--\ref{ha-s:6}).
  
  In each step, one could verify the signs directly, keeping track of
  \begin{enumarabic}
  \item the sign in the definition~\eqref{eq:def-Phi-n} of~\(\Phi\),
  \item the sign given by~\citehomog{eq.~\eqrefhomog{eq:twc-composition-sign}} that arises
    from the composition of the twisting cochains~\(\Phi\) and~\(\Phi\otimes1\),
  \item the sign that accounts for distributing composition of maps over tensor products as in~\eqref{eq:convention-eqKS},
  \item the product of the signs that appears when the first argument in each term \(E_{i_{s}}(A;c_{\bullet})\)
    is split into its factors as in~\eqref{eq:E-prod-E}
    where \(A\) is a summand of some~\(\Phi_{(j_{t})}(a_{\bullet}\otimes b_{\bullet})\) given by~\eqref{eq:def-Phi-n}, and
  \item the product of the signs appearing each time a term~\(E_{l}(E_{j_{t}}(a_{t};b_{\bullet});c_{\bullet})\)
    is expanded according to the associativity rule~\eqref{eq:formula-Ek-El}.
  \end{enumarabic}

  Alternatively, we can argue as follows:
  Since \(\Phi'\) is a twisting cochain, the corresponding family satisfies the defining equation~\citehomog{eq.~\eqrefhomog{eq:tw-fam-2}}.
  We observe that the terms~\(W\) and~\(W'\) share exactly one term in their differentials
  with respect to the \(\kk\)-basis for~\(\hgaop(n)\) described in \Cref{rem:hgaop-basis}.
  Moreover, this common term is not produced by the differential of any other term appearing in~\(\Phi'_{(n)}(a_{\bullet}\otimes b_{\bullet}\otimes c_{\bullet})\),
  nor by any term appearing on the right-hand side of the defining equation. 
  Hence the two terms in question must cancel out, which leads to the claimed sign rule.
\end{proof}

\section{Homotopy commutativity}
\label{sec:def-hc}

{

\def\a{\Fchoose{b}{a}}
\def\b{\Fchoose{a}{b}}
\def\i{\Fchoose{j}{i}}
\def\j{\Fchoose{i}{j}}
\def\p{\Fchoose{q}{p}}
\def\q{\Fchoose{p}{q}}
\def\s{\Fchoose{t}{s}}
\def\t{\Fchoose{s}{t}}

Let \(\Phi\colon A\otimes A\Rightarrow A\) be the shm map constructed in \Cref{sec:ai-map} or, more generally for the moment,
any shm map extending the multiplication in~\(A\) such that \Cref{thm:BA-product-Phi} holds. Then
\begin{equation}
  \Phi_{(2)}(a\otimes1,1\otimes b) + (-1)^{\deg{a}\deg{b}}\,\Phi_{(2)}(1\otimes b,a\otimes 1) = a\cupone b
\end{equation}
for all~\(a\),~\(b\in A\), \cf~\cite[Prop.~4.8]{Munkholm:1974}.\footnote{%
  The second~``\(a\)'' appearing in the argument of~\(\Phi\) in~\cite[Prop.~4.8]{Munkholm:1974} should read~``\(\bar a\)''.
  We also remark that the twisting cochain condition~\citehomog{eq.~\eqrefhomog{eq:def-tw-cochain}}
  implies \(\Phi_{(2)}(a\otimes 1,1\otimes b)=0\).}
(For our~\(\Phi\) this can be read off from~\eqref{eq:def-Phi-2}.)
Now assume that \(h\) is an shm homotopy from~\(\Phi\) to~\(\Phi\circ T\),
as required by property~\ref{shc-com} of an shc algebra. A straightforward computation shows that
\begin{align}
  a\cuptwo b &= (-1)^{\deg{a}\deg{b}}\,h_{(2)}(1\otimes b,a\otimes 1) - h_{(2)}(a\otimes 1,1\otimes b) \\
  \notag &\qquad + (-1)^{\deg{a}}\, a\cupone h_{(1)}(1\otimes b) + h_{(1)}(a\otimes 1)\cupone b
\end{align}
is a \(\cuptwo\)-product for~\(A\) in the sense that it satisfies \eqref{eq:cuptwo}.
As remarked in \Cref{sec:extended-hga}, a non-trivial Gerstenhaber bracket in~\(H^{*}(A)\)
is an obstruction to the existence of a \(\cuptwo\)-product and therefore to the homotopy commutativity of~\(\Phi\).

In order to proceed, we put as an additional assumption in this section that \(A\) be extended.
Let us define \(\hc_{(0)}=\eta_{A}\) and
\begin{multline}
  \label{eq:def-hc}
  \hc_{(n)}(a_{\bullet}\otimes b_{\bullet}) \eqKS
  \!\! \sum_{\i_{1}+\dots+\i_{n}=n}\!\!  E_{\i_{1}}(\b_{1};\a_{\bullet})\cdots E_{\i_{n}}(\b_{n};\a_{\bullet}) \\
  - \sum
  E_{\j_{1}}(\a_{1};\b_{\bullet})\cdots E_{\j_{\p}}(\a_{\p};\b_{\bullet}) \,
  F_{kl}(a_{\bullet};b_{\bullet}) \,
  E_{\i_{1}}(\b_{\p+1};\a_{\bullet})\cdots E_{\i_{\q}}(\b_{n};\a_{\bullet})
\end{multline}
for~\(n\ge1\).
The first sum is over all decompositions of~\(n\) into~\(n\) non-negative integers.
The second sum is over all positive integers~\(p\),~\(q\),~\(k\),~\(l\)
and all non-negative integers~\(\j_{1}\),~\dots,~\(\j_{\p}\), \(\i_{1}\),~\dots,~\(\i_{\q}\) such that
\begin{equation}
  \p+\q = n,
  \quad\quad\quad
  \forall\; 1\le \s\le \p \qquad \j_{1}+\dots+\j_{\s} < \s
\end{equation}
and
\begin{equation}
  \label{eq:def-hc-cond-2}
  \j_{1}+\dots+\j_{\p}+k = \p, 
  \quad\quad\quad
  \i_{1}+\dots+\i_{\q}+l = \q. 
\end{equation}
Omitting the argument \(a_{\bullet}\otimes b_{\bullet}\), the formula for~\(\hc\) looks as follows in small degrees.
\begin{align}
  \hc_{(1)} &\eqKS
E_{{1}} ( \b_{{1}};\a_{{1}} ),
  \\
  \hc_{(2)} &\eqKS
\b_{{1}}\, E_{{2}} ( \b_{{2}};\a_{{1}},\a_{{2}} )
+ E_{{1}} ( \b_{{1}};\a_{{1}} ) \, E_{{1}} ( \b_{{2}};\a_{{2}} )
+ E_{{2}} ( \b_{{1}};\a_{{1}},\a_{{2}} ) \, \b_{{2}}
\\ \notag &\qquad - \a_{{1}}\,  F_{{1,1}} ( a_{{1}};b_{{2}} ) \, \b_{{2}},
\displaybreak[1] \\ 
\hc_{(3)} &\eqKS
\b_{{1}}\, \b_{{2}}\, E_{{3}} ( \b_{{3}};\a_{{1}},\a_{{2}},\a_{{3}} )
+ \b_{{1}}\, E_{{1}} ( \b_{{2}};\a_{{1}} ) \, E_{{2}} ( \b_{{3}};\a_{{2}},\a_{{3}} ) 
\\ \notag &\qquad + \b_{{1}}\, E_{{2}} ( \b_{{2}};\a_{{1}},\a_{{2}} ) \, E_{{1}} ( \b_{{3}};\a_{{3}} ) 
+ \b_{{1}}\, E_{{3}} ( \b_{{2}};\a_{{1}},\a_{{2}},\a_{{3}} ) \, \b_{{3}}
\\ \notag &\qquad + E_{{1}} ( \b_{{1}};\a_{{1}} ) \, \b_{{2}}\, E_{{2}} ( \b_{{3}};\a_{{2}},\a_{{3}} ) 
+ E_{{1}} ( \b_{{1}};\a_{{1}} ) \, E_{{1}} ( \b_{{2}};\a_{{2}} ) \, E_{{1}} ( \b_{{3}};\a_{{3}} ) 
\\ \notag &\qquad + E_{{1}} ( \b_{{1}};\a_{{1}} ) \, E_{{2}} ( \b_{{2}};\a_{{2}},\a_{{3}} ) \, \b_{{3}}
+ E_{{2}} ( \b_{{1}};\a_{{1}},\a_{{2}} ) \, \b_{{2}}\, E_{{1}} ( \b_{{3}};\a_{{3}} ) 
\\ \notag &\qquad + E_{{2}} ( \b_{{1}};\a_{{1}},\a_{{2}} ) \, E_{{1}} ( \b_{{2}};\a_{{3}} ) \, \b_{{3}}
+ E_{{3}} ( \b_{{1}};\a_{{1}},\a_{{2}},\a_{{3}} ) \, \b_{{2}}\, \b_{{3}}
\\ \notag &\qquad - \a_{{1}}\,  F_{{1,1}} ( a_{{1}};b_{{2}} ) \, \b_{{2}}\, E_{{1}} ( \b_{{3}};\a_{{3}} )
- \a_{{1}}\,  F_{{1,1}} ( a_{{1}};b_{{2}} ) \, E_{{1}} ( \b_{{2}};\a_{{3}} ) \, \b_{{3}}
\\ \notag &\qquad - \a_{{1}}\, E_{{1}} ( \a_{{2}};\b_{{1}} ) \,  F_{{1,1}} ( a_{{2}};b_{{3}} ) \, \b_{{3}}
- \a_{{1}}\, \a_{{2}}\,  F_{{2,1}} ( a_{{1}},a_{{2}};b_{{3}} ) \, \b_{{3}}
\\ \notag &\qquad - \a_{{1}}\, F_{{1,2}} ( a_{{1}};b_{{2}},b_{{3}} ) \, \b_{{2}}\, \b_{{3}}.
\end{align}

\begin{proposition}
  \label{thm:hga-shc-homcom}
  Assume that the hga~\(A\) is extended.
  The maps~\(\hc_{(n)}\) assemble to an shm homotopy~\(\hc\) from~\(\Fchoose{\Phi\circ T}{\Phi}\) to~\(\Fchoose{\Phi}{\Phi\circ T}\).
\end{proposition}

Note that \(\hc\circ T\) is a homotopy in the other direction.

\begin{proof}
  This is a yet another lengthy direct verification, see \Cref{sec:hga-shc-homcom}.
  It is helpful to observe the following:
  Consider a term appearing in the second sum of~\eqref{eq:def-hc}, and let \(1\le m\le n\).
  Then \(\a_{m}\) appears in the leading group of \(E\)-terms if and only if \(\b_{m}\) appears in the same group or in the \(F\)-term.
  Equivalently, \(\b_{m}\) appears in the trailing group of \(E\)-terms if and only if \(\a_{m}\) appears in the same group or in the \(F\)-term.

  This in particular shows that \(\hc\) satisfies the normalization condition for twisting homotopy families
  because it is impossible for any term in the second sum and any~\(m\)
  that \(\a_{m}\) appears before the \(F\)-term and \(\b_{m}\) after it.
  That the first sum vanishes if some~\(\a_{m}=1\) is clear.
\end{proof}

Using the homotopy~\(\hc\),
we can generalize the formula for the \(\cupone\)-product on the bar construction
given by Kadeishvili~\cite[Prop.~2]{Kadeishvili:2003} for~\(\kk=\Z_{2}\).
Earlier, Baues~\cite[\S 2.9]{Baues:1998} obtained the dual formula
for the cobar construction~\(\Omega\,C(X)\) of a \(1\)-reduced simplicial set~\(X\) and any~\(\kk\)
(without using the surjection operad explicitly).

\begin{corollary}
  \label{thm:cup1-BA}
  The composition~\(\BB\hc\,\shuffle_{A,A}\) is
  a coalgebra homotopy from the product with commuted factors to the regular product on~\(\BB A\).
  The associated twisting cochain homotopy~\(\FF\) is given by
  \begin{equation*}
    \FF\bigl([a_{1}|\dots|a_{k}]\otimes[b_{1}|\dots|b_{l}]\bigr) =
    \begin{cases}
      1 & \text{if \(k=l=0\),} \\
      \mp F_{kl}(a_{\bullet};b_{\bullet}) & \text{if \(k\ge1\) and \(l\ge1\),} \\
      0 & \text{otherwise,}
    \end{cases}
  \end{equation*}
  where the ``\(\mp\)'' indicates a minus sign
  combined with the sign from \citehomog{eq.~\eqrefhomog{eq:def-tw-h-fam}}.
\end{corollary}

\begin{proof}
  Assume \(n=k+l>0\) and consider a term
  \begin{equation}
    \cc = \pm \bigl[\, a'_{1}\otimes b'_{1} \bigm| \ldots \bigm| a'_{n}\otimes b'_{n} \,\bigr]
  \end{equation}
  appearing in~\(\shuffle([a_{1}|\dots|a_{k}]\otimes[b_{1}|\dots|b_{l}])\).
  Any such term is mapped to~\(0\) by the first sum in the definition of~\(\hc_{(n)}\).
  Assume that \(\a'_{i}\) occurs in the leading group of \(E\)-term in a non-zero contribution to the second sum.
  Then \(\b'_{i}\) appears also in the leading group or in the \(F\)-term.
  The normalization conditions for the \(E\)-operations and the \(F\)-operations imply \(\b'_{i}\ne1\), so that we have \(\a'_{i}=1\).
  By looking at the trailing group of \(E\)-term, we similarly conclude \(\b'_{i}=1\). Hence \(p=k\) and~\(q=l\)
  in the definition of~\(\hc_{(n)}\),
  \begin{equation}
    \cc = \pm \Fchoose{\bigl[\, a_{1}\otimes 1 \bigm| \ldots \bigm| a_{k}\otimes 1 \bigm| 1\otimes b_{1} \bigm| \ldots \bigm| 1\otimes b_{l} \,\bigr]}
                      {\bigl[\, 1\otimes b_{1} \bigm| \ldots \bigm| 1\otimes b_{l} \bigm| a_{1}\otimes 1 \bigm| \ldots \bigm| a_{k}\otimes 1 \,\bigr]},
  \end{equation}
  and all variables~\(i_{t}\) and~\(j_{s}\) are \(0\).
\end{proof}

}

\section{Polynomial algebras}
\label{sec:poly}

Let \(A\) and~\(B\) be augmented dgas, and let \(\bbb\lhd B\) be a differential ideal.
Recall from~\citehomog{eq.~\eqrefhomog{eq:def-shm-strict}}
that an shm map~\(f\colon A\Rightarrow B\) is called \newterm{\(\bbb\)-strict}
if we have \(f_{(n)}\equiv 0\pmod\bbb\) for all components with~\(n\ge2\).
Similarly, a homotopy~\(h\) between~\(f\) and another shm map~\(g\) is called \newterm{\(\bbb\)-trivial}
if \(h_{(n)}\equiv 0\pmod\bbb\) for~\(n\ge1\), see~\citehomog{eq.~\eqrefhomog{eq:h-trivial-def}}.

We write \(\kk[x]\) for a polynomial algebra on a generator~\(x\) of even degree.
We start with an analogue of the first part of~\cite[Prop.~6.2]{Munkholm:1974}.
The second part will be addressed by \Cref{thm:poly1-to-dga-shc}.

\begin{proposition}
  \label{thm:poly1-Ai-strict}
  Let \(f\),~\(g\colon \kk[x]\Rightarrow A\) be \(\aaa\)-strict shm~maps.
  If there is a~\(b\in\aaa\) such that \(db=f_{(1)}(x)-g_{(1)}(x)\), then \(f\) and~\(g\) are homotopic via an \(\aaa\)-trivial homotopy.
\end{proposition}

\begin{proof}
  We assume first that \(g\) is strict with~\(g(x)=f_{(1)}(x)=\vcentcolon a\).
  It is a direct calculation to verify that an \(\aaa\)-trivial homotopy from~\(f\) to~\(g\) is given by the family
  \begin{equation}
    h_{(n)}(x^{k_{1}},\dots,x^{k_{n}}) =
    (-1)^{n-1}\!\!\sum_{k'+k''=k_{n}-1}\!\! f_{(n+1)}(x^{k_{1}},\dots,x^{k_{n-1}},x^{k'},x)\,a^{k''}
  \end{equation}
  for~\(n\ge1\), see \Cref{sec:poly1-Ai-strict}. The decomposition~\(k'+k''=k_{n}-1\) is into non-neg\-a\-tive integers.
  Note that \(h_{(n)}(x^{k_{\bullet}})\) vanishes for any~\(n\ge1\) if \(k_{n}\le1\)
  or \(k_{m}=0\) for some~\(m<n\). This in particular shows that \(h_{(n)}\) satisfies
  the normalization condition for twisting homotopy families.

  As a consequence of this and \citehomog{Lemma~\refhomog{thm:tw-h-equiv-rel}},
  we can assume that both~\(f\) and~\(g\) are strict in order to prove the general case. Then
  \begin{equation}
    h(x^{k}) = \!\! \sum_{k'+k''=k-1} \!\! f(x)^{k'} b\,g(x)^{k''}
  \end{equation}
  is an algebra homotopy from~\(f\) to~\(g\)
  (in the sense of~\cite[\S 1.11]{Munkholm:1974}) taking values in~\(\aaa\).
  It gives rise to an \(\aaa\)-trivial shm homotopy~\(\tilde h\) from~\(f\) to~\(g\)
  with~\(\tilde h_{(1)}=h\) and~\(\tilde h_{(n)}=0\) for~\(n\ge2\).
\end{proof}

The following is a variant of~\cite[Lemma~7.3]{Munkholm:1974} with an explicit homotopy.
Recall that any hga is canonically an shc algebra in the sense
of Stasheff--Halperin, and in the sense of Munkholm if it is extended.
We refer to~\citehomog{Sec.~\refhomog{sec:shc}} for the definition
of an (\(\aaa\)-natural) shc map.

\begin{proposition}
  \label{thm:poly1-to-dga-shc}
  Let \(A\) be an 
  hga such that all operations~\(E_{k}\), \(k\ge2\), take values in a common ideal~\(\aaa\lhd A\).
  Let \(a\in A\) be a cocycle of even degree
  and assume that there is a~\(b\in\aaa\) such that \(db=E_{1}(a;a)\).
  Then the dga map
  \begin{equation*}
    f\colon \kk[x] \to A,
    \quad
    x^{k}\mapsto a^{k}
  \end{equation*}
  is an \(\aaa\)-natural shc map.
\end{proposition}

\begin{proof}
  \def\xx{x^{k_{\bullet}}\otimes x^{l_{\bullet}}}
  We have to show that the dga map
  \begin{equation}
    \kk[x]\otimes\kk[x] \xrightarrow{\mu_{\kk[x]}} \kk[x]\stackrel{f}{\longrightarrow} A
  \end{equation}
  is homotopic to the shm~map
  \begin{equation}
    \kk[x]\otimes\kk[x] \xrightarrow{f\otimes f} A\otimes A\xRightarrow{\,\Phi_{A}\,} A
  \end{equation}
  via an \(\aaa\)-strict shm homotopy. Such a homotopy from the latter map to the former is given by the following twisting homotopy family,
  where we write \(\xx\) for the sequence of arguments~\(x^{k_{1}}\otimes x^{l_{1}}\),~\dots,~\(x^{k_{n}}\otimes x^{l_{n}}\):
  \begin{align}
    h_{(1)}(\xx) &= 0, \\
    h_{(2)}(\xx) &= \sum_{\substack{l'+l''=\\l_{1}-1}} a^{k_{1}}\,E_{2}(a^{k_{2}};a^{l'},a)\,a^{l''+l_{2}} \\
      \notag &\quad - \sum_{\substack{k'+k''=\\k_{2}-1}}\;\sum_{\substack{l'+l''=\\l_{1}-1}} a^{k_{1}+k'+l'} b\,a^{k''+l''+l_{2}}, \\
    \label{hh-1}
    h_{(n)}(\xx) &= \sum\;\sum_{\substack{l'+l''=\\l_{n-1}-1}}
      E_{i_{1}}(a^{k_{1}};a^{l_{\bullet}})\cdots E_{i_{n-1}}(a^{k_{n-1}};a^{l_{\bullet}}) \\*
      \notag &\qquad\qquad {}\cdot E_{i_{n}+1}(a^{k_{n}};\dots,a^{l_{n-2}},a^{l'},a)\,a^{l''+l_{n}} \\
    \notag
    &\quad - \sum\;\sum_{\substack{k'+k''=\\k_{n}-1}}\;\sum_{\substack{l'+l''=\\l_{n-1}-1}}
    E_{i_{1}}(a^{k_{1}};a^{l_{\bullet}})\cdots E_{i_{n-1}}(a^{k_{n-1}};a^{l_{\bullet}}) \\
      \notag &\qquad\qquad 
      {}\cdot a^{k'+l'} b\,a^{k''+l''+l_{n}}\\
    \notag
      &\quad - \sum\;\sum_{\substack{k'+k''=\\k_{n}-1}}\;\sum_{\substack{l'+l''=\\l_{n-2}-1}}
      E_{i_{1}}(a^{k_{1}};a^{l_{\bullet}})\cdots E_{i_{n-2}}(a^{k_{n-2}};a^{l_{\bullet}}) \\
      \notag &\qquad\qquad {}\cdot E_{i_{n-1}}(a^{k_{n-1}};\dots,a^{l_{n-3}},a^{l'})\,a^{k'} b\,a^{k''+l''+l_{n-1}+l_{n}}
  \end{align}
  for~\(n\ge3\).
  The first sum in the first group of~\eqref{hh-1} extends over
  all decompositions~\(n-1=i_{1}+\dots+i_{n}\) into \(n\)~non-negative integers such that
  \begin{equation}
    \forall\; 1\le s \le n\quad i_{1}+\dots+i_{s} < s.
  \end{equation}
  and the first sums in the other two groups of~\eqref{hh-1}
  similarly over all decompositions~\(n-2=i_{1}+\dots+i_{n-1}\) into \(n-1\)~non-negative integers such that
  \begin{equation}
    \forall\; 1\le s \le n-1\quad i_{1}+\dots+i_{s} < s.
  \end{equation}
  The decompositions into~\(k'+k''\) and~\(l'+l''\) also involve non-negative integers.
  Note that the formula for~\(n=2\) is the same as the one for~\(n\ge3\)
  with the sum over~\(l'+l''=l_{n-2}-1\) omitted.
  Also observe that \(b\) is of even degree~\(2\deg{a}-2\) and that \(i_{n}+1\ge2\)
  in the first group of~\eqref{hh-1} so that \(h\) is indeed \(\aaa\)-trivial.
  The verification of the homotopy property is a lengthy computation, see \Cref{sec:poly1-to-dga-shc}.

  That the normalization condition for twisting homotopy families is satisfied
  follows by direct inspection: Assume that \(k_{i}=l_{i}=0\) for some~\(i\).
  It is clear that each term in the sum for~\(n=2\)
  vanishes if \(k_{2}=0\) or~\(l_{1}=0\).
  Similarly, each term in the first two sums for~\(n\ge3\) vanishes
  if \(k_{n}=0\) or~\(l_{n-1}=0\) or~\(l_{m}=0\) for~\(m\le n-2\).
  For each term in the third sum we get this for~\(k_{n}=0\) or~\(l_{n-2}=0\) or~\(l_{m}=0\) for~\(m\le n-3\).
  In the case~\(k_{n-1}=0\) we finally use that the corresponding \(E\)-term has another argument since \(i_{n-1}\ge1\).
\end{proof}

\section{Concluding remark}

It would certainly be desirable to have a more conceptual proof of \Cref{thm:hga-shc}
than our explicit construction.
(\Cref{rem:ez} might indicate a first step in this direction.)
We point out, however, that the operad~\(\hgaop\) governing hgas has non-trivial homology
in the degrees we are interested in. In fact, since \(\hgaop(n)\) models
the configuration space of \(n\)~points in the plane,
its homology \(H_{k}(\hgaop(n))\) is non-zero for all~\(0\le k<n\), see
the references given in \Cref{rem:hgaop-basis}
as well as~\cite{Sinha:2013}.
The Gerstenhaber bracket is a non-trivial element in~\(H_{1}(\hgaop(n))\) for~\(n\ge2\),
and its iterations give non-zero elements of higher degree.
  
A direct way to see this is the following:
The Hochschild cohomology of an algebra~\(A\) is an algebra over~\(H(\hgaop)\).
If \(A\) is commutative, then \(\HH^{0}(A)=A\), and \(\HH^{1}(A)=\Der(A)\) are the derivations of~\(A\).
The Gerstenhaber bracket of~\(a\in A\) and~\(D\in\Der(A)\) is given by~\([D,a]=D(a)\in A\),
\cf~\cite[Props.~19,~22, Example~52]{Belmans:2018}.
If \(A=\kk[x_{1},\dots,x_{n}]\) is a polynomial algebra, then the expression
in \(n\)~elements from~\(\HH^{*}(A)\) involving a \(k\)-fold iterated bracket,
\begin{equation}
  \bigr[ \partial_{1}, \ldots \bigl[\partial_{k-1}, [\partial_{k}, x_{1}\cdots x_{k}] \bigr] \ldots \bigr]\cdot\underbrace{1\cdots 1}_{\mathclap{n-k-1\text{ factors}}} = 1,
\end{equation}
shows \(H_{k}(\hgaop(n))\ne0\) for any~\(0\le k< n\).

\appendix

\section{Proof of Proposition \ref{thm:hga-shc-tw}}
\label{sec:hga-shc-tw}

In this and the following appendices we outline several proofs that are elementary,
but lengthy computations. The main difficulties are to distinguish between the many cases to consider
and to keep track of the signs.
We write ``\textbf{X}~\(\to\)~\textbf{Y}'' to indicate
that the terms for case~\textbf{X} cancel with or result in the terms for case~\textbf{Y}.

\medbreak

\noindent\textbf{Terms produced by~\(d(\Phi_{(n)})\)}
\medskip
\begin{caselist}
\item \(b_{j}\)-variable moved out of an \(E(a_{i};...)\)-term to the left

  In this case we have \(i>1\).
  \begin{caselist}
  \item \(j<i-1\)
    \label{phi:12}
    \cancelswith{phi:21}
  \item \(j=i-1\)
    \label{phi:11}
    \cancelswith{phiphi:1}
  \end{caselist}
\item \(b\)-variable moved out of an \(E(a_{i};...)\)-term to the right
  \begin{caselist}
  \item \(i<n\)
    \label{phi:21}
    \cancelswith{phi:12}
  \item \(i=n\)
    \label{phi:22}
    \cancelswith{phi1:2}

    In this case we have \(j=n-1\).
  \end{caselist}
\item Two \(b\)-variables multiplied together in an \(E\)-term
  \label{phi:3}
  \cancelswith{phi1:1}

  In this case the corresponding \(a\)-variables are further to the left.
\end{caselist}

\bigbreak
\noindent\textbf{Terms appearing in~\(\Phi_{(n-1)}(\dots,a_{i}a_{i+1}\otimes b_{i}b_{i+1},\dots)\)}
\medskip

\begin{caselist}[resume]
\item \(i<n-1\)
  \label{phi1:1}
  \cancelswith{phi:3}
\item \(i=n-1\)
  \label{phi1:2}
  \cancelswith{phi:22}
\end{caselist}

\bigbreak
\noindent\textbf{Terms appearing in~\(\Phi_{(k)}\,\Phi_{(n-k)}\)}

\medskip

\begin{caselist}[resume]
\item all such terms
  \label{phiphi:1}
  \cancelswith{phi:11}
\end{caselist}

\section{Proof of Proposition \ref{thm:hga-shc-homass}}
\label{sec:hga-shc-homass}

We write \(\Phi'=\Phi\circ(\Phi\otimes1)\) and \(\Phi''=\Phi\circ(1\otimes\Phi)\).
Because of the recursive definition of the sign for each term appearing in~\(\ha\),
we first pair up the terms appearing in the equation under consideration,
\begin{multline}
  d(\ha_{(n)})(a_{\bullet}\otimes b_{\bullet}\otimes c_{\bullet}) \eqKS \\
  \sum_{k=1}^{n-1}(-1)^{k}\,\ha_{(n-1)}(a_{\bullet}\otimes b_{\bullet}\otimes c_{\bullet},a_{k}a_{k+1}\otimes b_{k}b_{k+1}\otimes c_{k}c_{k+1},a_{\bullet}\otimes b_{\bullet}\otimes c_{\bullet}) \\
  + \sum_{k=0}^{n}\Bigl( \Phi'_{(k)}(a_{\bullet}\otimes b_{\bullet}\otimes c_{\bullet})\,\ha_{(n-k)}(a_{\bullet}\otimes b_{\bullet}\otimes c_{\bullet}) \\
  - (-1)^{k}\,\ha_{(k)}(a_{\bullet}\otimes b_{\bullet}\otimes c_{\bullet})\, \Phi''_{(n-k)}(a_{\bullet}\otimes b_{\bullet}\otimes c_{\bullet}) \Bigr).
\end{multline}
In a second step we show that the signs work out the way they should.

\subsection{Pairing the terms}
\label{sec:hga-shc-homass-pairs}

We assume \(n\ge2\). Recall that we write \(\jamax\) for the maximal index~\(j\) such that \(b_{j}\) does not appear in a \(bc\)-product.

\medskip

\noindent\textbf{Terms produced by~\(d(\ha_{(n)})\)}
\medskip
\begin{caselist}
\item Terms coming from an \(a_{i}\)-term
    
  In this case we have \(i>1\).
  \begin{caselist}
  \item First argument moved out to the left
    \begin{caselist}
    \item \(b_{j}\)-term moved out to the left
      
      In this case the previous term is an \(a_{i-1}\)-term.
      \begin{caselist}
      \item \(j<i-1\)
	\label{ha:1111}
        \cancelswith{ha:1211}
      \item \(j=i-1\)
	\label{ha:1112}

        In this case either \(c_{j}c_{j+1}\) appears in some \(c\)-product or \(j=\jamax\).
        \begin{caselist}
	\item \(c_{j}c_{j+1}\) appears in a \(c\)-product
          \label{ha:11121}

          Note that this \(c\)-product may contain other factors.
          \begin{caselist}
	  \item \(c_{j}c_{j+1}\) appears in a \(c\)-product inside a \(b\)-term that is argument to an \(a\)-term
            \label{ha:111211}
            \cancelswith{ha:1413}
          \item \(c_{j}c_{j+1}\) appears in a \(c\)-product inside a \(bc\)-product that is argument to an \(a\)-term
            \label{ha:111215a}
	    \cancelswith{ha:14231}
          \item \(c_{j}c_{j+1}\) appears in a \(c\)-product that is argument to an \(a\)-term
            \label{ha:111213}
	    \cancelswith{ha:1333}
          \item \(c_{j}c_{j+1}\) appears in a \(c\)-product inside a top-level \(b\)-term before the \(a_{n}\)-term
            \label{ha:111215}
	    \cancelswith{ha:23}
	  \item \(c_{j}c_{j+1}\) appears in a \(c\)-product inside final \(bc\)-product\\ \hspace*{4ex}
            \label{ha:111216}
	    \cancelswith{ha:331}
	  \end{caselist}
          \item \(j=\jamax\)
   
            In this case \(c_{\jamax}c_{\jamax+1}\) does not appear in a \(c\)-product.
	    \begin{caselist}
	    \item \(b_{\jamax+1}\) is first \(b\)-variable in \(bc\)-product inside an \(a\)-term
              \begin{caselist}
	      \item \(b_{\jamax+1}\) is the only \(b\)-variable in this \(bc\)-product\\ \hspace*{4ex}
                \label{ha:1112211}
                \cancelswith{ha:1313212}
              \item \(b_{\jamax+1}\) is not the only \(b\)-variable in this \(bc\)-product\\ \hspace*{4ex}
                \label{ha:1112212}
                \cancelswith{ha:131221}
              \end{caselist}
	    \item \(b_{\jamax+1}\) is first \(b\)-variable in final \(bc\)-product
              \begin{caselist}
	      \item \(b_{\jamax+1}\) is not the only \(b\)-variable in this \(bc\)-product\\ \hspace*{4ex}
                \label{ha:1112221}
		\cancelswith{ha:12131}

                This case means \(\jamax<n-1\).
	      \item \(b_{\jamax+1}\) is the only \(b\)-variable in this \(bc\)-product
                \label{ha:1112222}
                \cancelswith{phiha:21}

		This case means \(\jamax=n-1\).
              \end{caselist}
            \end{caselist}
        \end{caselist}
      \end{caselist}
    \item \(bc\)-product moved out to the left
      
      Let's write the final \(c\)-variable of the \(bc\)-product as \(c_{j}\).
      \begin{caselist}
      \item \(j<i-1\)
	\label{ha:1121}
        \cancelswith{ha:1221}
      \item \(j=i-1\)
	\label{ha:1122}
        \cancelswith{haphi:2}

        All \(c_{k}\) with~\(k>j\) appear as single arguments, not inside a proper \(c\)-product.
      \end{caselist}
    \item \(c\)-product moved out to the left
      
      Let's write the final \(c\)-variable as \(c_{j}\).
      \begin{caselist}
      \item Previous term is \(a_{i-1}\)-term
	\label{ha:1131}
        \cancelswith{ha:1231}

	In this case we have \(j<i-1\).
      \item Previous term is (top-level) \(b_{i-1}\)-term
	\begin{caselist}
        \item \(j<i-1\)
          \label{ha:11321}
          \cancelswith{ha:22}
	\item \(j=i-1\)
          \label{ha:11322}
          \cancelswith{phiha:1}

          Since \(b_{j}\) is top-level, we cannot have \(j=\jamax\).
        \end{caselist}
      \end{caselist}
    \end{caselist}
    
  \item Last argument moved out to the right
    \begin{caselist}
    \item \(b_{j}\)-term moved out to the right
      \begin{caselist}
      \item Next term is \(a_{i+1}\)-term
	\label{ha:1211}
        \cancelswith{ha:1111}
      \item Next term is top-level \(b_{i}\)-term inside \(ab\)-product
	\label{ha:1212}
        \cancelswith{ha1:5}

	In this case we have \(j=i-1\), and \(c_{i-1}c_{i}\) appears in a \(c\)-product (ending in~\(c_{i}\)).
      \item Next term is first \(b\)-term of final \(bc\)-product
	
	In this case we have \(i=n\) and \(j=\jamax\).
	\begin{caselist}
        \item \(\jamax=1\)
          \label{ha:12133}
          \cancelswith{haphi:11}
	\item \(\jamax>1\)
          \begin{caselist}
	  \item \(c_{\jamax-1}c_{\jamax}\) does not appear in a \(c\)-product (inside the final \(bc\)-product)
            \label{ha:12131}
            \cancelswith{ha:1112221}
	  \item \(c_{\jamax-1}c_{\jamax}\) appears in a \(c\)-product (inside the final \(bc\)-product)\\ \hspace*{4ex}
            \label{ha:12132}
            \cancelswith{ha:330}
          \end{caselist}
        \end{caselist}
      \end{caselist}
    \item \(bc\)-product moved out to the right
      
      In this case the next term cannot be a top-level \(b\)-term.
      \begin{caselist}
      \item Next term is \(a_{i+1}\)-term
	\label{ha:1221}
        \cancelswith{ha:1121}
      \item Next term is first \(b\)-term of final \(bc\)-product
	\label{ha:1223}
        \cancelswith{ha:3122}
      \end{caselist}
    \item \(c\)-product moved out to the right
      \begin{caselist}
      \item Next term is \(a_{i+1}\)-term
	\label{ha:1231}
        \cancelswith{ha:1131}
      \item Next term is top-level \(b_{i}\)-term inside \(ab\)-product
	\label{ha:1232}
        \cancelswith{ha:21}
      \item Next term is first \(b\)-term of final \(bc\)-product
	\label{ha:1233}
        \cancelswith{ha:311}
      \end{caselist}
    \end{caselist}
  \item Two arguments multiplied together
    \begin{caselist}
    \item First argument is a \(b_{j}\)-term
      \begin{caselist}
      \item Second argument is a \(b_{j+1}\)-term
	\label{ha:1311}
        \cancelswith{ha1:4}

	In this case we have \(c_{j}c_{j+1}\) inside some \(c\)-product.
      \item Second argument is a \(bc\)-product

	In this case we have \(j=\jamax\).
        \begin{caselist}
	\item \(\jamax=1\)
          \label{ha:13121}
          \cancelswith{haphi:122}
        \item \(\jamax>1\)
	  \begin{caselist}
          \item \(b_{\jamax-1}\) appears in a top-level \(b\)-term
            \label{ha:131221}
	    \cancelswith{ha:1112212}

            In this case \(c_{\jamax-1}c_{\jamax}\) does not appear in any \(c\)-product.
          \item \(b_{\jamax-1}\) does not appear in a top-level \(b\)-term
            \label{ha:131222}
	    \cancelswith{ha:14230}

	    In this case \(c_{\jamax-1}c_{\jamax}\) appears in a \(c\)-product inside the given \(bc\)-product.
          \end{caselist}
        \end{caselist}
      \item Second argument is a \(c\)-product

	Let's write \(c_{k}\) for the final \(c\)-variable in the \(c\)-product.
        \begin{caselist}
	\item \(k < j\)
          \label{ha:13131}
          \cancelswith{ha:1412}
	\item \(k = j\)
          \begin{caselist}
	  \item \(c\)-product is the single variable~\(c_{j}\)
            \begin{caselist}
            \item \(k=j=1\)
              \label{ha:1313211}
              \cancelswith{haphi:121}
            \item \(k=j>1\)
              \label{ha:1313212}
              \cancelswith{ha:1112211}

              In this case \(b_{j-1}\) is top-level.
            \end{caselist}
	  \item \(c\)-product is proper \(c\)-product
            \label{ha:131322}
            \cancelswith{ha:1422212}
          \end{caselist}

          In this case we have \(j=k=\jamax\).
        \end{caselist}
      \end{caselist}
    \item First argument is a \(bc\)-product
      \label{ha:132}
      \cancelswith{ha:142122}

      In this case the second argument is also a \(bc\)-product.
    \item First argument is a \(c\)-product

      Let's write \(c_{k}\) for the last variable appearing in the \(c\)-product.
      \begin{caselist}
      \item Second argument is a \(b\)-term
	\label{ha:1331}
        \cancelswith{ha:1411}
      \item Second argument is a \(bc\)-product
	\label{ha:1332}
        \cancelswith{ha:14211}
      \item Second argument is a \(c\)-product
        \label{ha:1333}
        \cancelswith{ha:111213}

	In this case we have \(k<\jamax\) since all~\(c_{l}\) with~\(l>\jamax\) appear in \(bc\)-products.
        Hence \(b_{k}\) appears in a top-level \(b\)-term.
      \end{caselist}
    \end{caselist}
  \item Terms produced by one of the arguments of the \(a\)-term
    \begin{caselist}
    \item Terms produced by a \(b\)-term as argument
      
      Let's write the \(b\)-variable as~\(b_{j}\).
      \begin{caselist}
      \item \(c\)-product moved out to the left
	\label{ha:1411}
	\cancelswith{ha:1331}
      \item \(c\)-product moved out to the right
	\label{ha:1412}
	\cancelswith{ha:13131}
      \item Two \(c\)-products multiplied together
	\label{ha:1413}
         \cancelswith{ha:111211}
    
         Let's write \(c_{k}\) for the final \(c\)-variable of the first \(c\)-product.
         Then \(k<j\le\jamax\), hence \(b_{k}\) appears in a top-level \(b\)-term.
      \end{caselist}
    \item Terms produced by a \(b\)-term inside a \(bc\)-product
      
      Let's write the \(b\)-variable as~\(b_{j}\).
      \begin{caselist}
      \item \(c\)-product moved out to the left
 
        Let's write the final \(c\)-variable as~\(c_{k}\).
	\begin{caselist}
        \item \(b\)-term is the first \(b\)-term in the \(bc\)-product
	  \label{ha:14211}
          \cancelswith{ha:1332}
        \item \(b\)-term is not the first \(b\)-term in the \(bc\)-product
          \begin{caselist}
          \item \(k<j-1\)
            \label{ha:142121}
            \cancelswith{ha:14221}
          \item \(k=j-1\)
	    \label{ha:142122}
            \cancelswith{ha:132}

            In this case the \(c\)-product is a single variable~\(c_{i-1}\).
          \end{caselist}
        \end{caselist}
 
      \item \(c\)-product moved out to the right
 
        Let's write the final \(c\)-variable as~\(c_{k}\).
	\begin{caselist}
          \item \(b_{j}\) is not the last \(b\)-variable in the \(bc\)-product
            \label{ha:14221}
            \cancelswith{ha:142121}
          \item \(b_{j}\) is the last \(b\)-variable in the \(bc\)-product

	    In this case we have \(k=j-1\).
            \begin{caselist}
	    \item \(\jamax<j-1\)
              \label{ha:142221}
              \cancelswith{ha1:662}

	      In this case the \(c\)-product is the single variable~\(c_{j-1}\). Also, \(b_{j-1}\) is part of the same \(bc\)-product
              for otherwise \(c_{n-1}\) would be at the end of the previous \(bc\)-product.
            \item \(\jamax=j-1\)
	      \label{ha:1422212}
              \cancelswith{ha:131322}
            \end{caselist}
        \end{caselist}
      \item Two \(c\)-products multiplied together
    
        Let's write \(c_{k}\) for the final \(c\)-variable of the first \(c\)-product.
        \begin{caselist}
	\item \(b_{k}\) appears in a \(b\)-term inside an \(a\)-term
          \label{ha:14230}
          \cancelswith{ha:131222}

          In this case we have \(k=\jamax\).
        \item \(b_{k}\) appears in a top-level \(b\)-term
          \label{ha:14231}
          \cancelswith{ha:111215a}
        \item \(b_{k}\) appears in a \(bc\)-product (inside some \(a\)-term)
          \label{ha:14232}
          \cancelswith{ha1:661}

          In this case the second \(c\)-product is the single variable~\(c_{k+1}\),
          and \(b_{k+1}\) appears in the same \(bc\)-product as~\(b_{k}\).
        \end{caselist}
      \end{caselist}
    \end{caselist}
  \end{caselist}
\item Terms coming from a top-level \(b\)-term in the \(ab\)-product
  \begin{caselist}
  \item \(c\)-product moved out to the left
    \label{ha:21}
    \cancelswith{ha:1232}
  \item \(c\)-product moved out to the right
    \label{ha:22}
    \cancelswith{ha:11321}
  \item Two \(c\)-products multiplied together
    \label{ha:23}
    \cancelswith{ha:111215}
    
    Let's write \(c_{j}\) for the final \(c\)-variable of the first \(c\)-product.
    Then \(j<i\), hence \(b_{j}\) appears in a top-level \(b\)-term.
  \end{caselist}
\item Terms coming from a \(b_{i}\)-term in the final \(bc\)-product
  %
  \begin{caselist}
  \item \(c\)-product moved out to the left

    Let's write the final \(c\)-variable as~\(c_{j}\).
    \begin{caselist}
    \item \(b_{i}\) is the first \(b\)-variable in the \(bc\)-product
      \label{ha:311}
      \cancelswith{ha:1233}
    \item \(b_{i}\) is not the first \(b\)-variable in the \(bc\)-product
      \begin{caselist}
      \item \(j<i-1\)
	\label{ha:3121}
        \cancelswith{ha:321}
      \item \(j=i-1\)
	\label{ha:3122}
        \cancelswith{ha:1223}

        In this case the \(c\)-product is a single variable~\(c_{i-1}\).
      \end{caselist}
    \end{caselist}
  \item \(c\)-product moved out to the right

    Let's write the final \(c\)-variable as~\(c_{j}\).
    \begin{caselist}
    \item \(i<n\)
      \label{ha:321}
      \cancelswith{ha:3121}
    \item \(i=n\)

      In this case we have \(j=n-1\).
      \begin{caselist}
      \item \(\jamax<n-1\)
	\label{ha:3221}
        \cancelswith{ha1:672}

        In this case the \(c\)-product is the single variable~\(c_{n-1}\).
	Also, \(b_{n-1}\) is part of the same \(bc\)-product
        for otherwise \(c_{n-1}\) would be at the end of the previous \(bc\)-product.
      \item \(\jamax=n-1\)
	\label{ha:3222}
        \cancelswith{phiha:22}
      \end{caselist}
    \end{caselist}
  \item Two \(c\)-products multiplied together
    
    Let's write \(c_{j}\) for the final \(c\)-variable of the first \(c\)-product.
    Then \(b_{j}\) appears either in a top-level \(b\)-term or in a \(bc\)-product.
    \begin{caselist}
    \item \(b_{j}\) appears in an \(a\)-term
      \label{ha:330}
      \cancelswith{ha:12132}

      In this case we have \(j=\jamax\).
    \item \(b_{j}\) appears in a top-level \(b\)-term
      \label{ha:331}
      \cancelswith{ha:111216}
    \item \(b_{j}\) appears in a \(bc\)-product
      \label{ha:332}
      \cancelswith{ha1:671}

      In this case the second \(c\)-product is the single variable~\(c_{j+1}\),
      and \(b_{j+1}\) appears in the same \(bc\)-product as~\(b_{j}\).
    \end{caselist}
  \end{caselist}
\end{caselist}

\bigbreak
\noindent\textbf{Terms appearing in~\(\ha_{(n-1)}(\dots,a_{i}a_{i+1}\otimes b_{i}b_{i+1}\otimes c_{i}c_{i+1},\dots)\)}
\medskip

\begin{caselist}[resume]
\item \(b_{i}b_{i+1}\) in a \(b\)-term that is argument to an \(a\)-term
  \label{ha1:4}
  \cancelswith{ha:1311}
\item \(b_{i}b_{i+1}\) in a top-level \(b\)-term in the \(ab\)-product
  \label{ha1:5}
  \cancelswith{ha:1212}
\item \(b_{i}b_{i+1}\) in a \(b\)-term appearing in a \(bc\)-product
  \begin{caselist}
  \item \(c_{i}c_{i+1}\) in a \(bc\)-product inside an \(a\)-term
    \begin{caselist}
    \item \(c_{i}c_{i+1}\) is argument to a \(b\)-term in the \(bc\)-product
      \label{ha1:661}
      \cancelswith{ha:14232}
    \item \(c_{i}c_{i+1}\) at the end of the \(bc\)-product
      \label{ha1:662}
      \cancelswith{ha:142221}
    \end{caselist}
  \item \(c_{i}c_{i+1}\) in final \(bc\)-product
    \begin{caselist}
    \item \(c_{i}c_{i+1}\) is argument to a \(b\)-term in the \(bc\)-product
      \label{ha1:671}
      \cancelswith{ha:332}
    \item \(c_{i}c_{i+1}\) at the end of the \(bc\)-product
      \label{ha1:672}
      \cancelswith{ha:3221}
    \end{caselist}
  \end{caselist}
\end{caselist}

\bigbreak
\noindent\textbf{Terms appearing in~\(\bigl(\Phi\circ(\Phi\otimes1)\bigr)_{(k)}\,\ha_{(n-k)}\)}
\medskip

\begin{caselist}[resume]
\item \(k<n\)
  \label{phiha:1}
  \cancelswith{ha:11322}
\item \(k=n\)
  \begin{caselist}
  \item \(b_{n-1}\) appears before \(a_{n}\)-term in~\(\Phi'\)
    \label{phiha:21}
    \cancelswith{ha:1112222}

    In this case the final \(c\)-product in~\(\Phi'\) is \(c_{n}\) only.
  \item \(b_{n-1}\) appears inside \(a_{n}\)-term in~\(\Phi'\)
    \label{phiha:22}
    \cancelswith{ha:3222}

    In this case \(\Phi'\) ends with a proper \(c\)-product.
  \end{caselist}
\end{caselist}

\bigbreak
\noindent\textbf{Terms appearing in~\(\ha_{(k)}\,\bigl(\Phi\circ(1\otimes\Phi)\bigr)_{(n-k)}\)}
\medskip

\noindent
Note that all \(b\)'s and \(c\)'s in~\(\Phi''\) appear inside \(bc\)-products with no proper \(c\)-products.
\begin{caselist}[resume]
\item \(k=0\) 
  \begin{caselist}
  \item There is only the final \(bc\)-product in~\(\Phi''\) (preceded by the product of all \(a\)-variables)
    \label{haphi:11}
    \cancelswith{ha:12133}
  \item There is a \(bc\)-product inside some \(a\)-term in~\(\Phi''\)
    \begin{caselist}
    \item The first \(bc\)-product is \(b_{1}c_{1}\)
      \label{haphi:121}
      \cancelswith{ha:1313211}
    \item The first \(bc\)-product contains \(b_{2}\)
      \label{haphi:122}
      \cancelswith{ha:13121}
    \end{caselist}
  \end{caselist}
\item \(k>0\)
  \label{haphi:2}
  \cancelswith{ha:1122}
\end{caselist}

\subsection{Checking the signs}
\label{sec:hga-shc-homass-signs}

\def\hapair#1#2{Pair~\ref{#1}~\(\leftrightarrow\)~\ref{#2}\ \kern0pt}
\medbreak
\begin{caselist}
\item \hapair{ha:1111}{ha:1211}
  \label{ha-s:1}
  \begin{caselist}
  \item \(\mu=j=\nu\)
  \item \(\mu=j<\nu\)
    
  \item \(\mu<j\le\nu\)
  \label{ha-s:1.3}
    \begin{caselist}
    \item \(b_{\mu}\) appears before \(a_{i-1}\)-term
      \begin{caselist}
      \item \(c_{\mu}\) appears before \(a_{i-1}\)-term
      \item \(c_{\mu}\) appears in \(c\)-product that is argument to \(a_{i-1}\)-term
      \item \(c_{\mu}\) appears inside a \(b\)-term that is argument to \(a_{i-1}\)-term
      \item \(c_{\mu}\) appears in \(c\)-product that is argument to \(a_{i}\)-term
      \item \(c_{\mu}\) appears inside \(b_{j}\)-term (that is argument to \(a_{i}\)-term)
      \item \(c_{\mu}\) appears inside another \(b\)-term or \(bc\)-product that is argument to \(a_{i}\)-term
      \item \(c_{\mu}\) appears after \(a_{i}\)-term
      \end{caselist}
    \item \(b_{\mu}\) appears inside \(a_{i-1}\)-term
      \begin{caselist}
      \item \(c_{\mu}\) appears in \(c\)-product that is argument to \(a_{i-1}\)-term
      \item \(c_{\mu}\) appears inside a \(b\)-term that is argument to \(a_{i-1}\)-term
      \item \(c_{\mu}\) appears in \(c\)-product that is argument to \(a_{i}\)-term
      \item \(c_{\mu}\) appears inside \(b_{j}\)-term (that is argument to \(a_{i}\)-term)
      \item \(c_{\mu}\) appears inside another \(b\)-term or \(bc\)-product that is argument to \(a_{i}\)-term
      \item \(c_{\mu}\) appears after \(a_{i}\)-term
      \end{caselist}
    \end{caselist}
  \end{caselist}
  
\item \hapair{ha:111211}{ha:1413} and the following pairs
  \label{ha-s:2}
  
  In~\ref{ha:111211}, \(b_{j}\)-term is the first argument to \(a_{i}\)-term.
  \begin{caselist}
  \item \(\mu=j<\nu\)

    This case is built into the recursive sign formula.
  \item \(\mu<j<\nu\)
    \begin{caselist}
    \item \(c_{\mu}\) appears before \(a_{i}\)-term
    \item \(c_{\mu}\) appears in \(c\)-product that is argument to \(a_{i}\)-term
    \item \(c_{\mu}\) appears inside \(b_{j}\)-term (that is argument to \(a_{i}\)-term)
    \item \(c_{\mu}\) appears inside another \(b\)-term or \(bc\)-product that is argument to \(a_{i}\)-term
    \item \(c_{\mu}\) appears after \(a_{i}\)-term
    \end{caselist}
  \end{caselist}
  
\item \hapair{ha:111215a}{ha:14231}
  
  Subsumed under case~\ref{ha-s:2}
  
\item \hapair{ha:111213}{ha:1333}
  \label{ha-s:4}
  
  Subsumed under case~\ref{ha-s:2}
  
\item \hapair{ha:111215}{ha:23}
  \label{ha-s:5}
  
  Subsumed under case~\ref{ha-s:2}
  
\item \hapair{ha:111216}{ha:331}
  \label{ha-s:6}
  
  Subsumed under case~\ref{ha-s:2}
  
\item \hapair{ha:1112211}{ha:1313212} and the following pair
  \label{ha-s:7}
  \begin{caselist}
  \item \(\mu=j=\nu\)
    \begin{caselist}
    \item \(b_{\nu+1}\)-term appears in \(a_{i}\)-term
    \item \(b_{\nu+1}\)-term appears in later \(a\)-term
    \end{caselist}
  \item \(\mu<j=\nu\)
    \begin{caselist}
    \item \(c_{\mu}\) appears before \(a_{i}\)-term
    \item \(c_{\mu}\) appears in \(c\)-product that is argument to \(a_{i}\)-term
    \item \(c_{\mu}\) appears inside \(b_{\nu}\)-term (that is argument to \(a_{i}\)-term)
    \item \(c_{\mu}\) appears inside \(bc\)-product starting with~\(b_{\nu+1}\) (if it is argument to \(a_{i}\)-term)
    \item \(c_{\mu}\) appears after \(a_{i}\)-term
    \end{caselist}
  \end{caselist}
  
\item \hapair{ha:1112212}{ha:131221}
  
  Subsumed under case~\ref{ha-s:7}
  
\item \hapair{ha:1112221}{ha:12131}
  \begin{caselist}
  \item \(\mu=j=\nu\)
  \item \(\mu<j=\nu\)
    \begin{caselist}
    \item \(c_{\mu}\) appears before \(a_{i}\)-term
    \item \(c_{\mu}\) appears in \(c\)-product that is argument to \(a_{i}\)-term
    \item \(c_{\mu}\) appears inside \(b_{\nu}\)-term (that is argument to \(a_{i}\)-term)
    \item \(c_{\mu}\) appears after \(a_{i}\)-term
    \end{caselist}
  \end{caselist}

\item \hapair{ha:1112222}{phiha:21}
  \begin{caselist}
  \item \(\mu=\nu\)
  \item \(\mu<\nu\)
    
    This cases uses \Cref{thm:sign-change-Y}.
    \begin{caselist}
    \item \(c_{\mu}\) appears before the \(a_{n}\)-term
    \item \(c_{\mu}\) appears in \(b_{n-1}\)-term
    \item \(c_{\mu}\) appears in \(c\)-product that is argument to \(a_{n}\)-term
    \item \(c_{\mu}\) appears in \(b_{n}\)-term
    \end{caselist}
  \end{caselist}
  
\item \hapair{ha:1121}{ha:1221}
  \label{ha-s:11}
  \begin{caselist}
  \item \(\mu=\nu\)
  \item \(\mu<\nu\)
    
    This case is analogous to~\ref{ha-s:1.3}\ (with some subcases omitted).
  \end{caselist}
  
\item \hapair{ha:1122}{haphi:2}
  \begin{caselist}
  \item \(\mu=\nu\)
  \item \(\mu<\nu\)
    \begin{caselist}
    \item \(c_{\mu}\) appears before \(a_{i}\)-term
    \item \(c_{\mu}\) appears in \(bc\)-product that is first argument to \(a_{i}\)-term
    \end{caselist}
  \end{caselist}
  
\item \hapair{ha:1131}{ha:1231}
  \label{ha-s:13}
  
  This case is analogous to case~\ref{ha-s:11}\
  (by considering a \(c\)-product as a ``\(bc\)-product with no \(b\)-terms and several trailing \(c\)-variables'')
  with the following additional case for the recursive step.
  \begin{caselist}
  \item \(c_{\mu}\) appears inside the \(c\)-product that is first argument to \(a_{i}\)-term
  \end{caselist}
  
\item \hapair{ha:11321}{ha:22}
  
  This case is analogous to case~\ref{ha-s:13}

\item \hapair{ha:11322}{phiha:1}

  Note that \(k\ne\mu\).
  \begin{caselist}
  \item \(k<\mu\le\nu\)
    \label{ha-s:15.1}
    \begin{caselist}
    \item \(\mu=\nu\)
    \item \(\mu<\nu\)
    \end{caselist}
  \item \(\mu<k<\nu\)

    Using \Cref{thm:sign-change-Y}, this case is recursively reduced to~\ref{ha-s:15.1}
  \end{caselist}

\item \hapair{ha:1212}{ha1:5}
  \begin{caselist}
  \item \(|J_{a}|=2\), that is, \(J_{a}=\{\mu,\nu\}\)
    \label{ha-s:16.1}

    Let's say that \(b_{\nu}\) appears in the \(a_{k}\)-term.
    \begin{caselist}
    \item \(c_{\mu}\) appears before \(a_{k}\)-term
    \item \(c_{\mu}\) appears in \(c\)-product that is argument to \(a_{k}\)-term before \(b_{\nu}\)-term
    \item \(c_{\mu}\) appears inside \(b_{\nu}\)-term
    \item \(c_{\mu}\) appears in \(c\)-product that is argument to \(a_{k}\)-term after \(b_{\nu}\)-term,
      but before any \(bc\)-product (if \(a_{k}\)-term has such arguments)
    \item \(c_{\mu}\) appears inside first \(bc\)-product that is argument to \(a_{k}\)-term (if \(a_{k}\)-term has such an argument)
    \end{caselist}
  \item \(|J_{a}|>2\)
    \begin{caselist}
    \item \(\mu=j<\nu-1\)
      \label{ha-s:16.2.1}

      Let \(k=\min(J_{a}\setminus\{\mu\})\). There are only top-level \(b\)-terms
      between the \(a_{k}\)-term and the \(a\)-term containing the \(b_{k}\)-term.
      Using the cases~\ref{ha-s:1}\ and~\ref{ha-s:13},
      we can take the \(b_{k}\)-term (together with the preceding \(c\)-products)
      out of the \(a\)-term and move it right after the \(a_{k}\)-term.
      The sign change is analogous to the recursive definition~\eqref{eq:ha-sign-def-rec} of the sign.
      We then use induction with \(J_{a}\setminus\{k\}\) instead of~\(J_{a}\) until we reach the case~\ref{ha-s:16.1}
    \item \(\mu<j\le\nu-1\)
      \begin{caselist}
      \item \(c_{\mu}\) appears before \(a_{i}\)-term
      \item \(c_{\mu}\) appears inside \(b_{i-1}\)-term
      \item \(c_{\mu}\) appears inside other \(b\)-term that is argument to \(a_{i}\)-term
      \item \(c_{\mu}\) appears in \(c\)-product that is argument to \(a_{i}\)-term
      \item \(c_{\mu}\) appears inside \(b_{i}\)-term
      \item \(c_{\mu}\) appears after \(b_{i}\)-term
      \end{caselist}
    \end{caselist}
  \end{caselist}

\item \hapair{ha:12133}{haphi:11}

\item \hapair{ha:12132}{ha:330}
  \begin{caselist}
  \item \(|J_{a}|=2\), that is, \(J_{a}=\{\mu,\nu\}\) with~\(\mu=\nu-1\)
    \begin{caselist}
    \item \(b_{\mu}\) appears before \(a_{n}\)-term
    \item \(b_{\mu}\) appears inside \(a_{n}\)-term
    \end{caselist}
  \item \(|J_{a}|>2\)
    \begin{caselist}
    \item \(c_{\mu}\) appears before \(a_{n}\)-term
    \item \(c_{\mu}\) appears inside \(b_{\nu}\)-term
    \item \(c_{\mu}\) appears inside other \(b\)-term that is argument to \(a_{n}\)-term
    \item \(c_{\mu}\) appears in \(c\)-product that is argument to \(a_{n}\)-term
    \item \(c_{\mu}\) appears in final \(bc\)-product
    \end{caselist}
  \end{caselist}

\item \hapair{ha:1223}{ha:3122}
  \label{ha-s:19}
  \begin{caselist}
  \item \(J_{a}=\{\nu\}\)
  \item \(|J_{a}|>1\)
    \label{ha-s:19.2}
    \begin{caselist}
    \item \(b_{\mu}\) appears before \(a_{n}\)-term
      \begin{caselist}
      \item \(c_{\mu}\) appears before \(a_{n}\)-term
      \item \(c_{\mu}\) appears in \(c\)-product that is argument to \(a_{n}\)-term
      \item \(c_{\mu}\) appears in \(b\)-term that is argument to \(a_{n}\)-term
      \item \(c_{\mu}\) appears in \(bc\)-product that is argument to \(a_{n}\)-term, but not the last argument
	\label{ha-s:19.2.1.4}
      \item \(c_{\mu}\) appears in final \(bc\)-product inside \(a_{n}\)-term (but not as trailing \(c\)-variable)
	\label{ha-s:19.2.1.5}
      \end{caselist}
    \item \(b_{\mu}\) appears inside \(a_{n}\)-term
      \begin{caselist}
      \item \(c_{\mu}\) appears in \(c\)-product that is argument to \(a_{n}\)-term
      \item \(c_{\mu}\) appears in \(b\)-term that is argument to \(a_{n}\)-term
      \item \(c_{\mu}\) appears in \(bc\)-product that is argument to \(a_{n}\)-term, but not the last argument
	\label{ha-s:19.2.2.3}
      \item \(c_{\mu}\) appears in final \(bc\)-product inside \(a_{n}\)-term (but not as trailing \(c\)-variable)
	\label{ha-s:19.2.2.4}
      \end{caselist}
    \end{caselist}
  \end{caselist}
  
\item \hapair{ha:1232}{ha:21}
  \label{ha-s:20}

  This case is analogous to~\ref{ha-s:19}\
  (again by considering a \(c\)-product as a ``\(bc\)-product with no \(b\)-terms and several trailing \(c\)-variables'')
  with the following additional subcases for the case~\ref{ha-s:19.2}\ 
  and with some subcases omitted.
  Note that \(\mu\ne i\).
  \begin{caselist}
  \item \(b_{\mu}\) appears before \(a_{i}\)-term
    \begin{caselist}
    \item \(c_{\mu}\) appears in final \(c\)-product inside \(a_{i}\)-term
    \item \(c_{\mu}\) appears in or after \(b_{i}\)-term
    \end{caselist}
  \item \(b_{\mu}\) appears inside \(a_{i}\)-term
    \begin{caselist}
    \item \(c_{\mu}\) appears in final \(c\)-product inside \(a_{i}\)-term
    \item \(c_{\mu}\) appears in or after \(b_{i}\)-term
    \end{caselist}
  \item \(b_{\mu}\) appears after \(b_{i}\)-term
    \label{ha-s:20.3}
  \end{caselist}
  
\item \hapair{ha:1233}{ha:311}
  
  This case is analogous to~\ref{ha-s:19}\ and~\ref{ha-s:20}
  with some subcases omitted.
  
\item \hapair{ha:1311}{ha1:4}

  We have \(\mu\le j\le\nu-1\).
  \begin{caselist}
  \item \(\mu=j=\nu-1\), that is, \(J_{a}=\{\nu-1,\nu\}\)
    \begin{caselist}
      \item \(c_{\mu}c_{\mu+1}\) is argument to \(a_{i}\)-term
      \item \(c_{\mu}c_{\mu+1}\) appears after \(a_{i}\)-term and before any \(bc\)-products
	\label{ha-s:22.1.2}
      \item \(c_{\mu}c_{\mu+1}\) appears in (first) \(bc\)-product inside \(a_{i}\)-term
      \item \(c_{\mu}c_{\mu+1}\) appears in (first) \(bc\)-product inside later \(a\)-term
      \item \(c_{\mu}c_{\mu+1}\) appears in final \(bc\)-product
	
	This case is analogous to~\ref{ha-s:22.1.2}
    \end{caselist}
  \item \(\mu=j<\nu-1\)

    Let \(k=\min(J_{a}\setminus\{\mu,\mu+1\})\).
    \begin{caselist}
    \item \(i\le k\)
      \label{ha-s:22.2.1}

      This case is analogous to~\ref{ha-s:16.2.1}\ We can take the \(b_{k}\)-term out of the \(a\)-term where it appears.
    \item \(i>k\)
      
      Using the cases~\ref{ha-s:1}\ and~\ref{ha-s:13},
      we can move the \(b_{\mu}\)-term and the \(b_{\mu+1}\)-term (together with the preceding \(c\)-products)
      to the \(a_{i-1}\)-term. This inductively reduces this case to~\ref{ha-s:22.2.1}
    \end{caselist}
  \item \(\mu<j\le\nu-1\)
    \begin{caselist}
    \item \(b_{\mu}\) appears before \(a_{i}\)-term
      \begin{caselist}
      \item \(c_{\mu}\) appears before \(a_{i}\)-term
      \item \(c_{\mu}\) appears inside \(a_{i}\)-term before \(b_{j}\)
      \item \(c_{\mu}\) appears inside \(b_{j}\)-term
      \item \(c_{\mu}\) appears inside \(b_{j+1}\)-term
      \item \(c_{\mu}\) appears after \(b_{j+1}\)-term
      \end{caselist}
    \item \(b_{\mu}\) appears inside \(a_{i}\)-term
      \begin{caselist}
      \item \(c_{\mu}\) appears inside \(a_{i}\)-term before \(b_{j}\)
      \item \(c_{\mu}\) appears inside \(b_{j}\)-term
      \item \(c_{\mu}\) appears inside \(b_{j+1}\)-term
      \item \(c_{\mu}\) appears after \(b_{j+1}\)-term
      \end{caselist}
    \end{caselist}
  \end{caselist}

\item \hapair{ha:13121}{haphi:122}
  \label{ha-s:23}

  This case uses \citehomog{eq.~\eqrefhomog{eq:twc-composition-sign}}
  for the sign of the composition of two shm maps
  and the convention~\eqref{eq:convention-eqKS} for our ``\(\eqKS\)'' notation.
  Note that the sign exponent~\(n-1\) in~\eqref{eq:def-Phi-n} leads to the sign exponent~\(n-1\)
  in~\((\Phi\circ(1\otimes\Phi))_{(n)}\). 

\item \hapair{ha:131222}{ha:14230}
  \begin{caselist}
  \item \(\mu<\nu-1\)
    \begin{caselist}
    \item \(c_{\mu}\) appears before \(a_{i}\)-term
    \item \(c_{\mu}\) appears in \(c\)-product that is argument to \(a_{i}\)-term before \(b_{\nu}\)
    \item \(c_{\mu}\) appears in \(c\)-product inside \(b_{\nu}\)-term
    \item \(c_{\mu}\) appears between~\(b_{\nu+1}\) and~\(c_{\nu}\)
    \end{caselist}
  \item \(\mu=\nu-1\)
    \begin{caselist}
    \item \(b_{\mu}\) appears before \(a_{i}\)-term
    \item \(b_{\mu}\) appears inside \(a_{i}\)-term
    \end{caselist}    
  \end{caselist}
  
\item \hapair{ha:13131}{ha:1412}
  \label{ha-s:25}
  \begin{caselist}
  \item \(\mu=j=\nu\)
  \item \(\mu=j<\nu\)
    \begin{caselist}
    \item \(c_{\mu}\) appears inside \(a_{i}\)-term
    \item \(c_{\mu}\) appears after \(a_{i}\)-term
    \end{caselist}
  \item \(\mu<j\le\nu\)
    \begin{caselist}
    \item \(b_{\mu}\) appears before \(a_{i}\)-term
      \begin{caselist}
      \item \(c_{\mu}\) appears before \(a_{i}\)-term
      \item \(c_{\mu}\) appears inside \(a_{i}\)-term before \(b_{\nu}\)-term
      \item \(c_{\mu}\) appears inside \(b_{\nu}\)-term
      \item \(c_{\mu}\) appears in \(c\)-product following \(b_{\nu}\)-term (together with~\(c_{k}\))
      \item \(c_{\mu}\) appears inside \(a_{i}\)-term after \(c_{k}\)
      \item \(c_{\mu}\) appears after \(a_{i}\)-term
      \end{caselist}
    \item \(b_{\mu}\) appears inside \(a_{i}\)-term
      \begin{caselist}
      \item \(c_{\mu}\) appears inside \(a_{i}\)-term before \(b_{\nu}\)-term
      \item \(c_{\mu}\) appears inside \(b_{\nu}\)-term
      \item \(c_{\mu}\) appears in \(c\)-product following \(b_{\nu}\)-term (together with~\(c_{k}\))
      \item \(c_{\mu}\) appears inside \(a_{i}\)-term after \(c_{k}\)
      \item \(c_{\mu}\) appears after \(a_{i}\)-term
      \end{caselist}
    \end{caselist}
  \end{caselist}
  
\item \hapair{ha:1313211}{haphi:121}

  This case is analogous to~\ref{ha-s:23}

\item \hapair{ha:131322}{ha:1422212}
  \begin{caselist}
  \item \(J_{a}=\{\mu=\nu-1,\nu\}\)
    \begin{caselist}
    \item \(b_{\mu}\) appears before \(a_{i}\)-term
    \item \(b_{\mu}\) appears inside \(a_{i}\)-term
    \end{caselist}
  \item \(|J_{a}|>2\)
    \begin{caselist}
    \item \(b_{\mu}\) appears before \(a_{i}\)-term
      \begin{caselist}
      \item \(c_{\mu}\) appears before \(a_{i}\)-term
      \item \(c_{\mu}\) appears inside \(a_{i}\)-term before \(b_{\nu}\)-term
      \item \(c_{\mu}\) appears inside \(b_{\nu}\)-term
      \item \(c_{\mu}\) appears inside \(a_{i}\)-term in \(c\)-product following \(b_{\nu}\)-term

	This case uses \ref{ha-s:25}
      \end{caselist}
    \item \(b_{\mu}\) appears inside \(a_{i}\)-term
      \begin{caselist}
      \item \(c_{\mu}\) appears inside \(a_{i}\)-term before \(b_{\nu}\)-term
      \item \(c_{\mu}\) appears inside \(b_{\nu}\)-term
      \item \(c_{\mu}\) appears inside \(a_{i}\)-term in \(c\)-product following \(b_{\nu}\)-term

	This case uses \ref{ha-s:25}
      \end{caselist}
    \end{caselist}
  \end{caselist}
  
\item \hapair{ha:132}{ha:142122}
  \label{ha-s:28}
  \begin{caselist}
  \item \(J_{a}=\{\nu\}\)
  \item \(|J_{a}|>1\)
    \begin{caselist}
    \item \(b_{\mu}\) appears before \(a_{i}\)-term
      \begin{caselist}
      \item \(c_{\mu}\) appears before \(a_{i}\)-term
      \item \(c_{\mu}\) appears inside \(a_{i}\)-term before first \(bc\)-product
      \item \(c_{\mu}\) appears inside first \(bc\)-product
      \end{caselist}
    \item \(b_{\mu}\) appears inside \(a_{i}\)-term
      \begin{caselist}
      \item \(c_{\mu}\) appears inside \(a_{i}\)-term before first \(bc\)-product
      \item \(c_{\mu}\) appears inside first \(bc\)-product
      \end{caselist}
    \end{caselist}
  \end{caselist}
  
\item \hapair{ha:1331}{ha:1411}
  \begin{caselist}
  \item \(J_{a}=\{\nu\}\)
  \item \(|J_{a}|>1\)
    
    Let's write \(b_{j}\) for the \(b\)-term following \(c_{k}\).
    \begin{caselist}
    \item \(b_{\mu}\) appears before \(a_{i}\)-term
      \begin{caselist}
      \item \(c_{\mu}\) appears before \(a_{i}\)-term
      \item \(c_{\mu}\) appears inside \(a_{i}\)-term before \(c\)-product containing \(c_{k}\)
      \item \(c_{\mu}\) appears in \(c\)-product containing \(c_{k}\)
      \item \(c_{\mu}\) appears in \(b_{j}\)-term
      \item \(c_{\mu}\) appears inside \(a_{i}\)-term after \(b_{j}\)-term
      \item \(c_{\mu}\) appears after \(a_{i}\)-term
      \end{caselist}
    \item \(b_{\mu}\) appears inside \(a_{i}\)-term before \(c\)-product containing \(c_{k}\)
      \begin{caselist}
      \item \(c_{\mu}\) appears inside \(a_{i}\)-term before \(c\)-product containing \(c_{k}\)
      \item \(c_{\mu}\) appears in \(c\)-product containing \(c_{k}\)
      \item \(c_{\mu}\) appears in \(b_{j}\)-term
      \item \(c_{\mu}\) appears inside \(a_{i}\)-term after \(b_{j}\)-term
      \item \(c_{\mu}\) appears after \(a_{i}\)-term
      \end{caselist}
    \item \(\mu=j\)

      This case uses \ref{ha-s:20}
    \end{caselist}
  \end{caselist}
  
\item \hapair{ha:1332}{ha:14211}

  This case is analogous to~\ref{ha-s:28}\ with the following additional cases.
  \begin{caselist}
  \item \(|J_{a}|>1\), \(b_{\mu}\) appears before \(a_{i}\)-term, \(c_{\mu}\) appears inside (formerly second) \(bc\)-product
  \item \(|J_{a}|>1\), \(b_{\mu}\) appears inside \(a_{i}\)-term, \(c_{\mu}\) appears inside (formerly second) \(bc\)-product
  \end{caselist}

\item \hapair{ha:142121}{ha:14221}
  \label{ha-s:31}

  Let's write the two \(b\)-variables in question as \(b_{j}\) and~\(b_{j+1}\).
  \begin{caselist}
  \item \(J_{a}=\{\nu\}\)
  \item \(|J_{a}|>1\)
    \begin{caselist}
    \item \(c_{\mu}\) appears before \(b_{j}\)-term
    \item \(c_{\mu}\) appears inside \(b_{j}\)-term
    \item \(c_{\mu}\) appears inside or after \(b_{j+1}\)-term
    \end{caselist}
  \end{caselist}
    
\item \hapair{ha:142221}{ha1:662}
  \begin{caselist}
  \item \(J_{a}=\{\nu\}\)
  \item \(|J_{a}|>1\)  
    \label{ha-s:32.2}
  \end{caselist}

\item \hapair{ha:14232}{ha1:661}
  \begin{caselist}
  \item \(J_{a}=\{\nu\}\)
  \item \(|J_{a}|>1\)  
    \label{ha-s:33.2}
  \end{caselist}
  
\item \hapair{ha:3121}{ha:321}

  This case is analogous to~\ref{ha-s:31}
  
\item \hapair{ha:3221}{ha1:672}
  \begin{caselist}
  \item \(J_{a}=\{\nu\}\)
  \item \(|J_{a}|>1\)  

    This case is analogous to~\ref{ha-s:32.2}
  \end{caselist}

\item \hapair{ha:3222}{phiha:22}
  \begin{caselist}
  \item \(\mu=\nu\)
  \item \(\mu<\nu\)
    
    This case uses \Cref{thm:sign-change-Y}.
    \begin{caselist}
    \item \(c_{\mu}\) appears before last argument of \(b_{n}\)-term
    \item \(c_{\mu}\) appears in \(c\)-product that is last argument of \(b_{n}\)-term
      \label{ha-s:36.2.2}
    \end{caselist}
  \end{caselist}
  
\item \hapair{ha:332}{ha1:671}
  \begin{caselist}
  \item \(J_{a}=\{\nu\}\)
  \item \(|J_{a}|>1\)  

    This case is analogous to~\ref{ha-s:33.2}
  \end{caselist}

\end{caselist}

\section{Proof of Proposition \ref{thm:hga-shc-homcom}}
\label{sec:hga-shc-homcom}

{

\def\a{\Fchoose{b}{a}}
\def\b{\Fchoose{a}{b}}
\def\i{\Fchoose{j}{i}}
\def\j{\Fchoose{i}{j}}

\noindent\textbf{Terms produced by~\(d(\hc_{(n)})\)}
\medskip

\noindent
Recall that there are leading \(E\)-terms if and only if there is an \(F\)-term.
If there is an \(F\)-term, the index of each \(\a\)-variable appearing in it is larger than the index of any \(\b\)-variables appearing in it.
\begin{caselist}
\item Terms produced by group of trailing \(E\)-terms

  The trailing \(E\)-terms may or may not be preceded by a leading \(E\)-group and an \(F\)-term.
  \begin{caselist}
  \item \(\a\)-variable moved out of a trailing \(E_{m}(\b_{\j};\dots)\)-term to the left
    \begin{caselist}
    \item This \(E\)-term is the first \(E\)-term of the trailing group
      \begin{caselist}
      \item There is an \(F\)-term
	\label{hc:1111}
        \cancelswith{hc:32}
      \item There is no \(F\)-term
	\begin{caselist}
          \item \(m=1\)
	    \label{hc:11121}
            \cancelswith{phihc:611}
          \item \(m>1\)
	    \label{hc:11122}
            \cancelswith{hc:361}
        \end{caselist}
      \end{caselist}
    \item This \(E\)-term is a later \(E\)-term of the trailing group
      \label{hc:112}
      \cancelswith{hc:121}
    \end{caselist}
  \item \(\a\)-variable moved out of a trailing \(E\)-term to the right
    \begin{caselist}
    \item \(E\)-term is not final \(E\)-term
      \label{hc:121}
      \cancelswith{hc:112}
    \item \(E\)-term is final \(E\)-term
      \label{hc:122}
      \cancelswith{hcphi:1}

      In this case we can split up the trailing \(E\)-terms in a unique way such that
      the second part is a valid term in~\(\Fchoose{\Phi}{\Phi'=\Phi_{(n-k)}\circ T}\)
    \end{caselist}
  \item Two \(\a\)-variables multiplied together in a trailing \(E\)-term
    \label{hc:13}
    \cancelswith{hc1:11}
  \end{caselist}

\item Terms produced by group of leading \(E\)-terms
  \begin{caselist}
  \item \(\b_{\j}\)-variable moved out of a leading \(E(\a_{\i};\dots)\)-term to the left

    In this case we have \(\i>1\).
    \begin{caselist}
    \item \(\j<\i-1\)
      \label{hc:211}
      \cancelswith{hc:222}
    \item \(\j=\i-1\)
      \label{hc:212}
      \cancelswith{phihc:62}
    \end{caselist}
  \item \(\b\)-variable moved out of a leading \(E\)-term to the right
    \begin{caselist}
    \item This \(E\)-term is the last \(E\)-term of the leading group
      \label{hc:221}
      \cancelswith{hc:33}
    \item This \(E\)-term is an earlier \(E\)-term of the leading group      
      \label{hc:222}
      \cancelswith{hc:211}
    \end{caselist}
  \item Two \(\b\)-variables multiplied together in a leading \(E\)-term
    \label{hc:23}
    \cancelswith{hc1:11}

    In this case the corresponding \(\a\)-variables appear further to the left.
  \end{caselist}
  
\item Terms produced by \(F\)-term
  \begin{caselist}
  \item \(E(\a_{\i};\dots)\)-term split off \(F\)-term to the left
    \label{hc:31}
    \cancelswith{hc:34}

    Note the first \(\b_{\j}\)variable of the trailing \(E\)-terms has index~\(\j\ge \i\).
  \item \(\a\)-variable moved out of \(F\)-term to the right
    \label{hc:32}
    \cancelswith{hc:1111}
  \item \(\b\)-variable moved out of \(F_{kl}(\a_{\i},\dots;b_{\j},\dots)\)-term to the left
    \label{hc:33}
    \cancelswith{hc:221}

    In this case we have \(l>1\), so that \(\j<\i-1\).
  \item \(E(\b_{\j};\dots)\)-term split off \(F\)-term to the right
    \label{hc:34}
    \cancelswith{hc:31}

    Note the last \(\a_{\i}\)variable of the leading \(E\)-terms has index~\(\i\ge \j\).
  \item \(F_{1l}\)-term converted into \(E_{l}(\a;\dots)\)-term with~\(l\ge1\)
    \begin{caselist}
    \item The next term is a \(\b\)-variable
      \label{hc:351}
      \cancelswith{phihc:612}
    \item The next term is an \(E_{m}(\b;\dots)\)-term with~\(m\ge1\)
      \label{hc:352}
      \cancelswith{hc:362}
    \end{caselist}
  \item \(F_{k1}\)-term converted into \(E_{k}(\b_{\j};\dots)\)-term with~\(k\ge1\)
    \begin{caselist}
    \item The previous \(E\)-term is an \(\a\)-variable
      \label{hc:361}
      \cancelswith{hc:11122}

      In this case the \(\a\)-variable is the only leading \(E\)-term because otherwise \(\b_{\j-1}\) would appear in the \(F\)-term.
    \item The previous \(E\)-term is an \(E_{m}(\a;\dots)\)-term with~\(m\ge1\)
      \label{hc:362}
      \cancelswith{hc:352}
    \end{caselist}
  \item Two \(\a\)-variables multiplied together
    \label{hc:37}
    \cancelswith{hc1:22}

    In this case the corresponding \(\b\)-variables occur in trailing \(E\)-terms
  \item Two \(\b\)-variables multiplied together
    \label{hc:38}
    \cancelswith{hc1:12}

    In this case the corresponding \(\a\)-variables occur in leading \(E\)-terms
  \end{caselist}
\end{caselist}

\bigbreak
\noindent\textbf{Terms appearing in~\(\hc_{(n-1)}(\dots,a_{i}a_{i+1}\otimes b_{i}b_{i+1},\dots)\)}
\medskip

\noindent
Either the repeated \(\a\)-terms appear in a leading \(E\)-term or the repeated \(\b\)-terms appear in a trailing \(E\)-term.
\begin{caselist}[resume]
\item The repeated \(\a\)-terms appear in a leading \(E\)-term
  \begin{caselist}
  \item The repeated \(\b\)-terms appear in a leading \(E\)-term
    \label{hc1:11}
    \cancelswith{hc:23}
  \item The repeated \(\b\)-terms appear in the \(F\)-term
    \label{hc1:12}
    \cancelswith{hc:38}
  \end{caselist}
\item The repeated \(\b\)-terms appear in a trailing \(E\)-term
  \begin{caselist}
  \item The repeated \(\a\)-terms appear in the trailing \(E\)-term
    \label{hc1:21}
    \cancelswith{hc:13}
  \item The repeated \(\a\)-terms appear in the \(F\)-term
    \label{hc1:22}
    \cancelswith{hc:37}
  \end{caselist}
\end{caselist}

\bigbreak
\noindent\textbf{Terms appearing in~\(\Fchoose{\Phi'}{\Phi}_{(m)}\,\hc_{(n-m)}\)}
\medskip

\Fchoose{\noindent Here \(\Phi'=\Phi\circ T\).}{}
\begin{caselist}[resume]
\item \(\hc_{(n-m)}\) does not contain an \(F\)-term
  \begin{caselist}
  \item \(m=1\)
    \label{phihc:611}
    \cancelswith{hc:11121}
  \item \(m>1\)
    \label{phihc:612}
    \cancelswith{hc:351}
  \end{caselist}
\item \(\hc_{(n-m)}\) contains an \(F\)-term
  \label{phihc:62}
  \cancelswith{hc:212}
\end{caselist}

\bigbreak
\noindent\textbf{Terms appearing in~\(\hc_{(m)}\,\Fchoose{\Phi}{\Phi'}_{(n-m)}\)}
\medskip
\Fchoose{}{\noindent Here \(\Phi'=\Phi\circ T\).}
\begin{caselist}[resume]
\item all such terms
  \label{hcphi:1}
  \cancelswith{hc:122}
\end{caselist}

}

\section{Proof of Proposition \ref{thm:poly1-Ai-strict}}
\label{sec:poly1-Ai-strict}

We assume \(n\ge1\).

\medbreak

\noindent\textbf{Terms produced by~\(d(h_{(n)})\)}
\medskip

\begin{caselist}
\item Two arguments of~\(f_{(n+1)}\) multiplied together
  \begin{caselist}
  \item At position~\(m\le n-2\) (if \(n\ge3\))
    \label{pg:11}
    \cancelswith{pg1:1}
  \item At position~\(m=n-1\) (if \(n\ge2\))
    \label{pg:12}
    \cancelswith{pg1:21}
  \item At position~\(m=n\)
    \begin{caselist}
    \item \(k''=0\)
      \label{pg:131}
      \cancelswith{fpg:2}
    \item \(k''>0\)
      \label{pg:132}
      \cancelswith{pg:222}
    \end{caselist}
  \end{caselist}
\item \(f_{(n+1)}\) split into two terms
  \begin{caselist}
  \item At position~\(m\le n-1\) (if \(n\ge2\))
    \label{pg:21}
    \cancelswith{fpg:1}
  \item At position~\(m=n\)
    \begin{caselist}
    \item \(k'=0\)
      \label{pg:221}
      \cancelswith{pgg:0}

      By the normalization condition, this term vanishes unless \(n=1\).
    \item \(k'>0\)
      \label{pg:222}
      \cancelswith{pg:132}
    \end{caselist}
  \end{caselist}
\end{caselist}

\bigbreak
\noindent\textbf{Terms appearing in~\(h_{(n-1)}(\dots,x^{k_{m}+k_{m+1}},\dots)\)}
\medskip

\begin{caselist}[resume]
\item \(m\le n-2\) (if \(n\ge3\))
  \label{pg1:1}
  \cancelswith{pg:11}
\item \(m=n-1\) (if \(n\ge2\))
  \begin{caselist}
  \item \(k''<k_{n}\)
    \label{pg1:21}
    \cancelswith{pg:12}
  \item \(k''\ge k_{n}\)
    \label{pg1:22}
    \cancelswith{pgg:1}
    \end{caselist}
\end{caselist}

\bigbreak
\noindent\textbf{Terms appearing in~\(f_{(m)}\,h_{(n-m)}\)}
\medskip

\begin{caselist}[resume]
\item \(m<n\) (if \(n\ge2\))
  \label{fpg:1}
  \cancelswith{pg:21}
\item \(m=n\)
  \label{fpg:2}
  \cancelswith{pg:131}
\end{caselist}

\bigbreak
\noindent\textbf{Terms appearing in~\(h_{(m)}\,g_{(n-m)}\)}
\medskip

\noindent
Since \(g\) is strict, the only non-zero contribution is for~\(m=n-1\).

\begin{caselist}[resume]
\item \(n=1\)
  \label{pgg:0}
  \cancelswith{pg:221}
\item \(n\ge2\)
  \label{pgg:1}
  \cancelswith{pg1:22}
\end{caselist}

\section{Proof of Proposition \ref{thm:poly1-to-dga-shc}}
\label{sec:poly1-to-dga-shc}

\def\omark{{}^{\boldsymbol*}}
\def\othree{\(\mathllap{\omark}\)\ }
\def\ofour{\(\mathllap{\omark}\omark\)\ }

We write \(F=\mu_{A}\,(f\otimes f)\) and \(g=f\,\mu_{\kk[x]}\).
There is nothing to show for~\(n=1\) since \(F_{(1)}=g\).
Below we assume \(n\ge2\). Terms appearing only for~\(n\ge3\) are marked ``\(\omark\)''
and those appearing only for~\(n\ge4\) are marked ``\(\omark\omark\)''.

\medskip

\noindent\textbf{Terms produced by~\(d(h_{(n)})\)}
\medskip

\begin{caselist}

\item Terms produced by the first group of sums
  \begin{caselist}
  \item Argument moved out of some~\(E(a^{k_{i}};\dots)\) to the left

    In this case with have \(i>1\).
    \begin{caselist}
    \item \othree \(i\le n-1\)

      In this case the argument is some~\(a^{l_{j}}\).
      \begin{caselist}
      \item \ofour \(j<i-1\)
	\label{hp:1111}
        \cancelswith{hp:121}
      \item \othree \(j=i-1\)
	\label{hp:1112}
        \cancelswith{Fh:11}
      \end{caselist}
    \item \(i=n\)
      \begin{caselist}
      \item \(i_{n}=1\)
	\label{hp:1121}
	\cancelswith{hp:2.4}
      \item \othree \(i_{n}>1\)
	\label{hp:1122}
	\cancelswith{hp:122}
 
        In this case the argument is some~\(a^{l_{j}}\).
      \end{caselist}
    \end{caselist}
    
    
  \item Argument moved out of some~\(E(a^{k_{i}};\dots)\) to the right
    \begin{caselist}
    \item \ofour \(i\le n-2\)
      \label{hp:121}
      \cancelswith{hp:1111}
    \item \othree \(i=n-1\)
      \label{hp:122}
      \cancelswith{hp:1122}
    \item \(i=n\)
      \label{hp:123}
      \cancelswith{hp:1332}

      We may assume \(l'\ge 1\) in this case.
    \end{caselist}

  \item Two arguments of some~\(E(a^{k_{i}};\dots)\) multiplied together
    \begin{caselist}
    \item \ofour Arguments are \(a^{l_{j}}\) and~\(a^{l_{j+1}}\) with~\(j\le n-3\)
      \label{hp:131}
      \cancelswith{hp1:11}
    \item \othree Arguments are \(a^{l_{n-2}}\) and~\(a^{l'}\)
      \label{hp:132}
      \cancelswith{hp1:121}

      In this case we have \(i_{n}>1\).
    \item Arguments are \(a^{l'}\) and~\(a\)
      \begin{caselist}
      \item \(l''=0\)
	\label{hp:1331}
	\cancelswith{Fh:2}
      \item \(l''>0\)
	\label{hp:1332}
	\cancelswith{hp:123}
      \end{caselist}
    \end{caselist}

  \end{caselist}
  
\item Terms produced by the second group of sums
  \begin{caselist}
  \item \othree Argument~\(a^{l_{j}}\) moved out of some~\(E(a^{k_{i}};\dots)\) to the left

    In this case with have \(i>1\).
    \begin{caselist}
    \item \ofour \(j<i-1\)
      \label{hp:211}
      \cancelswith{hp:221}
    \item \othree \(j=i-1\)
      \label{hp:212}
      \cancelswith{Fh:21}
    \end{caselist}
  \item Argument moved out of some~\(E(a^{k_{i}};\dots)\) to the right
    \begin{caselist}
    \item \ofour \(i\le n-2\)
      \label{hp:221}
      \cancelswith{hp:211}
    \item \othree \(i=n-1\)
      \begin{caselist}
      \item \othree \(i_{n-1}=1\)
	\label{hp:2221}
        \cancelswith{hp1:231}
      \item \othree \(i_{n-1}>1\)
	\label{hp:2222}
        \cancelswith{hp1:221}
      \end{caselist}
    \end{caselist}
  \item \ofour Two arguments~\(a^{l_{j}}\) and~\(a^{l_{j+1}}\) of some~\(E(a^{k_{i}};\dots)\) multiplied together
    \label{hp:23}
    \cancelswith{hp1:21}

    In this case we have \(j\le n-3\).
  \item Term produced by~\(b\)
    \label{hp:2.4}
    \cancelswith{hp:1121}

    We have \(\sum_{k'+k''=k_{n}-1}a^{k'}\,E_{1}(a;a)\,a^{k''}=E_{1}(a^{k_{n}};a)\).
  \end{caselist}

\item \othree Terms produced by the third group of sums
  \begin{caselist}
  \item \othree Argument moved out of some~\(E(a^{k_{i}};\dots)\) to the left

    In this case with have \(i>1\).
    \begin{caselist}
    \item \ofour  \(i\le n-2\)

      In this case the argument is some~\(a^{l_{j}}\).
      \begin{caselist}
      \item \ofour \(j<i-1\)
	\label{hp:3111}
        \cancelswith{hp:321}
      \item \ofour \(j=i-1\)
	\label{hp:3112}
        \cancelswith{Fh:3}
      \end{caselist}
    \item \othree \(i=n-1\)
      \begin{caselist}
      \item \othree \(i_{n-1}=1\)
	\label{hp:3121}
	\cancelswith{hp:3231}
      \item \othree \(i_{n-1}>1\)
	\label{hp:3122}
	\cancelswith{hp:322}
 
        In this case the argument is some~\(a^{l_{j}}\).
      \end{caselist}
    \end{caselist}

  \item \othree Argument moved out of some~\(E(a^{k_{i}};\dots)\) to the right

    \begin{caselist}
    \item \ofour \(i\le n-3\)
      \label{hp:321}
      \cancelswith{hp:3111}
    \item \othree \(i=n-2\)
      \label{hp:322}
      \cancelswith{hp:3122}
    \item \othree \(i=n-1\)
      \begin{caselist}
      \item \othree \(i_{n-1}=1\)
	\label{hp:3231}
        \cancelswith{hp:3121}
      \item \othree \(i_{n-1}>1\)
	\label{hp:3232}
        \cancelswith{hp1:222}
      \end{caselist}
    \end{caselist}
    
  \item \ofour Two arguments of some~\(E(a^{k_{i}};\dots)\) multiplied together
    \begin{caselist}
    \item \ofour Arguments are \(a^{l_{j}}\) and~\(a^{l_{j+1}}\) with~\(j\le n-4\)
      \label{hp:331}
      \cancelswith{hp1:31}
    \item \ofour Arguments are \(a^{l_{n-3}}\) and~\(a^{l'}\)
      \label{hp:332}
      \cancelswith{hp1:321}

      In this case we have \(i_{n-1}>1\).
    \end{caselist}

  \item \othree Term produced by~\(b\)
    \label{hp:34}
    \cancelswith{hp1:132}

    We have \(\sum_{k'+k''=k_{n}-1}a^{k'}\,E_{1}(a;a)\,a^{k''}=E_{1}(a^{k_{n}};a)\).
  \end{caselist}

\end{caselist}

\bigbreak
\noindent\textbf{Terms appearing in~\(h_{(n-1)}(\dots,a^{k_{i}+k_{i+1}}\otimes a^{l_{i}+l_{i+1}},\dots)\)}
\medskip

\begin{caselist}[resume]

\item \othree Terms produced by the first group of sums
  \begin{caselist}
  \item \ofour \(i\le n-3\)
    \label{hp1:11}
    \cancelswith{hp:131}
  \item \othree \(i=n-2\)

    We have either \(l'\ge l_{n-2}\) or \(l''\ge l_{n-1}\).
    \begin{caselist}
    \item \othree \(l'\ge l_{n-2}\)
      \label{hp1:121}
      \cancelswith{hp:132}
    \item \othree \(l''\ge l_{n-1}\)
      \label{hp1:122}
      \cancelswith{hp1:133}
    \end{caselist}
  \item \othree \(i=n-1\)

    We decompose the term~\(E_{i_{n-1}+1}(a^{k_{n-1}+k_{n}};\dots,a^{l_{n-3}},a^{l'},a)\) into the following three kinds of terms.
    \begin{caselist}
    \item \othree \(E_{i_{n-1}+1}(a^{k_{n-1}};\dots,a^{l_{n-3}},a^{l'},a)\,a^{k_{n}}\)
      \label{hp1:131}
      \cancelswith{hg:1}
    \item \othree \(E_{i_{n-1}}(a^{k_{n-1}};\dots,a^{l_{n-3}},a^{l'})\,E_{1}(a^{k_{n}};a)\)
      \label{hp1:132}
      \cancelswith{hp:34}
    \item \othree \(E_{i_{n-1}}(a^{k_{n-1}};a^{l_{\bullet}})\,E_{i_{n}+1}(a^{k_{n}};\dots,a^{l_{n-3}},a^{l'},a)\) with~\(i_{n}\ge1\)
      \label{hp1:133}
      \cancelswith{hp1:122}
    \end{caselist}
  \end{caselist}
  
\item \othree Terms produced by the second group of sums
  \begin{caselist}
  \item \ofour \(i\le n-3\)
    \label{hp1:21}
    \cancelswith{hp:23}
  \item \othree \(i=n-2\)
    
    We have either \(l'\ge l_{n-2}\) or \(l''\ge l_{n-1}\).
    \begin{caselist}
    \item \othree \(l'\ge l_{n-2}\)
      \label{hp1:221}
      \cancelswith{hp:2222}
    \item \othree \(l''\ge l_{n-1}\)
      \label{hp1:222}
      \cancelswith{hp:3232}
    \end{caselist}
  \item \othree \(i=n-1\)

    We have either \(k'\ge k_{n-1}\) or \(k''\ge k_{n}\).
    \begin{caselist}
    \item \othree \(k'\ge k_{n-1}\)
      \label{hp1:231}
      \cancelswith{hp:2221}
    \item \othree \(k''\ge k_{n}\)
      \label{hp1:232}
      \cancelswith{hg:2}
    \end{caselist}
  \end{caselist}

\item \ofour Terms produced by the third group of sums
  \begin{caselist}
  \item \ofour \(i\le n-4\)
    \label{hp1:31}
    \cancelswith{hp:331}
  \item \ofour \(i=n-3\)

    We have either \(l'\ge l_{n-3}\) or \(l''\ge l_{n-2}\).
    \begin{caselist}
    \item \ofour \(l'\ge l_{n-3}\)
      \label{hp1:321}
      \cancelswith{hp:332}
    \item \ofour \(l''\ge l_{n-2}\)
      \label{hp1:322}
      \cancelswith{hp1:332}
    \end{caselist}
  \item \ofour \(i=n-2\)

    We decompose the term~\(E_{i_{n-2}}(a^{k_{n-2}+k_{n-1}};\dots,a^{l_{n-4}},a^{l'})\) into the following two kinds of terms.
    \begin{caselist}
    \item \ofour \(E_{i_{n-2}}(a^{k_{n-2}};\dots,a^{l_{n-4}},a^{l'})\,a^{k_{n-1}}\)
      \label{hp1:331}
      \cancelswith{hp1:341}
    \item \ofour \(E_{i_{n-2}}(a^{k_{n-2}};a^{l_{\bullet}})\,E_{i_{n-1}}(a^{k_{n-1}};\dots,a^{l_{n-4}},a^{l'})\) with~\(i_{n-1}\ge1\)
      \label{hp1:332}
      \cancelswith{hp1:322}
    \end{caselist}
  \item \ofour \(i=n-1\)

    We have either \(k'\ge k_{n-1}\) or \(k''\ge k_{n}\).
    \begin{caselist}
    \item \ofour \(k'\ge k_{n-1}\)
      \label{hp1:341}
      \cancelswith{hp1:331}
    \item \ofour \(k''\ge k_{n}\)
      \label{hp1:342}
      \cancelswith{hg:3}
    \end{caselist}
  \end{caselist}

\end{caselist}

\bigbreak
\noindent\textbf{Terms appearing in~\(F_{(m)}\,h_{(n-m)}\)}
\medskip

\noindent
Recall that \(h_{(1)}=0\).
\begin{caselist}[resume]
\item \othree \(1\le m\le n-2\)
  \begin{caselist}

  \item \othree Terms produced by the first group of sums in~\(h\)
    \label{Fh:11}
    \cancelswith{hp:1112}
    
  \item \othree Terms produced by the second group of sums in~\(h\)
    \label{Fh:21}
    \cancelswith{hp:212}
    
  \item \ofour Terms produced by the third group of sums in~\(h\)
    \label{Fh:3}
    \cancelswith{hp:3112}

  \end{caselist}
  
\item \(m=n\)
  \label{Fh:2}
  \cancelswith{hp:1331}
\end{caselist}

\bigbreak
\noindent\textbf{Terms appearing in~\(h_{(n-1)}\,g\)}
\medskip

\begin{caselist}[resume]
\item \othree Terms produced by the first group of sums in~\(h\)
  \label{hg:1}
  \cancelswith{hp1:131}
  
\item \othree Terms produced by the second group of sums in~\(h\)
  \label{hg:2}
  \cancelswith{hp1:232}
  
\item \ofour Terms produced by the third group of sums in~\(h\)
  \label{hg:3}
  \cancelswith{hp1:342}
  
\end{caselist}

\end{document}